\documentclass{article}

\usepackage[english]{babel}

\usepackage[letterpaper,top=2cm,bottom=2cm,left=3cm,right=3cm,marginparwidth=1.75cm]{geometry}

\usepackage{amsmath}
\usepackage{graphicx}
\usepackage[colorlinks=true, allcolors=blue]{hyperref}
\usepackage{amsthm}

\theoremstyle{definition}
\newtheorem{dfn}{Definition}[section]
\newtheorem{prop}[dfn]{Proposition}
\newtheorem{lem}[dfn]{Lemma}
\newtheorem{thm}[dfn]{Theorem}
\newtheorem{cor}[dfn]{Corollary}
\newtheorem{rem}[dfn]{Remark}

\newcommand{\sgn}{\mathop{\mathrm{sgn}}\nolimits}
\DeclareMathOperator*{\supp}{supp}

\usepackage{bbm}
\usepackage{amssymb}
\usepackage{amsfonts}
\usepackage{mathtools}
\usepackage{bm}
\usepackage{graphicx}
\usepackage{mathrsfs}

\title{Malliavin Calculus on the Clifford Algebra}
\author{Takayoshi Watanabe\thanks{Mathematical Institute, Graduate School of Science, Tohoku University, Sendai, Japan\\
Email address: takayoshi.watanabe.q5@dc.tohoku.ac.jp}}
\date{}

\begin{document}
\maketitle

\begin{abstract}
We deal with Malliavin calculus on the $L^2$ space of the $W^*$-algebra generated by fermion fields (the Clifford algebra). First, we verify the product formula for multiple integrals in It\^o-Clifford calculus, which is It\^o calculus on the Clifford algebra. Using this product formula, we can define the derivation operator and the divergence operator. Anti-symmetric Malliavin calculus thus constructed has properties similar to those of usual Malliavin calculus. The derivation operator and the divergence operator satisfy the canonical anti-commutation relations, and the divergence operator serves as an extension of the It\^o-Clifford stochastic integral, satisfying the Clark-Ocone formula. Subsequently, using this calculus, we consider the concentration inequality, the logarithmic Sobolev inequality, and the fourth-moment theorem. As for the logarithmic Sobolev inequality, only a weaker result is obtained. On the other hand, as for the concentration inequality, we obtain results that are almost similar to those of the usual case; moreover, the results concerning the fourth-moment theorem imply that, unlike in the case of the usual Brownian motion, convergence in distribution cannot be deduced from the convergence of the fourth moment.
\end{abstract}

\section{Introduction}
The present paper explores whether Malliavin calculus can function in the framework of non-commutative probability theory beyond bosons. This exploration opens up the possibility of applying techniques established in stochastic calculus to non-commutative probability theory.

Various results have already been established for bosons. Here, by a boson, we mean a triple consisting of the symmetric Fock space, its creation and annihilation operators $a^*$ and $a$. It is known that there exists an isomorphism from the $L^2$ space of functionals on the paths of Brownian motion to the symmetric Fock space over $L^2(\mathbb{R}_+)$. The fact that $B_t = a^*(\mathbbm{1}_{[0,t]}) + a(\mathbbm{1}_{[0,t]})$, where $\mathbbm{1}$ is the indicator function, reproduces the moments of Brownian motion is also well known. These facts lead to a variety of results. One of the existing results in this framework is Hudson-Parthasarathy theory \cite{Par92,Mey95}, which serves as It\^o calculus in bosons. Just as It\^o calculus yields the solutions to the diffusion equations, Hudson-Parthasarathy theory makes the solutions to the Lindblad equations. Subsequently, Malliavin calculus was extended in the form including Hudson-Parthasarathy theory, leading to non-causal stochastic calculus \cite{Lin93,Mey95}, and reaching a more generalized framework presented in \cite{FLS01}.

In the context of non-commutative calculus beyond bosons, there exists It\^o-Clifford calculus \cite{BSW1}, which serves as It\^o calculus for fermions. This calculus considers the creation and annihilation operators $b^*$ and $b$ on the anti-symmetric Fock space, where $\Psi_t = b^*(\mathbbm{1}_{[0,t]}) + b(\mathbbm{1}_{[0,t]})$ plays the role of Brownian motion. This calculus also develops martingale theory for this process $\Psi_t$.

Therefore, we construct Malliavin calculus for fermions to investigate whether Malliavin calculus can also be functional beyond bosons.

In section~\ref{section_B_preliminary} of this paper, we provide a brief introduction to the relationship between Brownian motion and bosons. In subsection~\ref{subsection_Malliavin}, we list the essential results from usual Malliavin calculus that will be needed. In subsection~\ref{subsection_BB_relation}, we explore how Brownian motion is realized within the framework of non-commutative probability theory using bosons.

In section~\ref{section_CAR algebra}, following the reference \cite{BSW1}, we introduce the results of It\^o-Clifford calculus that are necessary for this paper. Especially, we show that most of the results presented in section~\ref{section_B_preliminary} have their correspondence in the Clifford algebra. Among these, it is crucial that the chaos expansion holds. The results we obtain in this study rely on this fact.

In subsection~\ref{subsection_Mcf}, we take the results introduced in section~\ref{section_CAR algebra} into account, define the derivation operator and the divergence operator on the Clifford algebra, and show that these operators have properties similar to those of usual Malliavin calculus. The derivation operator and the divergence operator satisfy the canonical anti-commutation relations. The divergence operator serves as an extension of the It\^o-Clifford stochastic integral and satisfies the Clark-Ocone formula. It should be noted that we recently found that the derivation operator had already been constructed in \cite{AMN21}, where the existence of the divergence operator was also suggested.

In subsection~\ref{subsection_application}, we apply the contents discussed in subsection~\ref{subsection_Mcf} to analyse the Clifford algebra. We consider the concentration inequality, the logarithmic Sobolev inequality, and the fourth-moment theorem. As for the logarithmic Sobolev inequality, only a weaker result is obtained. On the other hand, as for the concentration inequality, we obtain results that are almost similar to those of the usual case; moreover, the results concerning the fourth-moment theorem imply that, unlike in the case of the usual Brownian motion, convergence in distribution cannot be deduced from the convergence of the fourth moment.

% 2
\section{Preliminaries About Brownian Motion}\label{section_B_preliminary}
% 2.1
\subsection{Outline of Malliavin Calculus}\label{subsection_Malliavin}
First, in this subsection, we provide a brief review of Malliavin calculus. The contents and notation are based on \cite{Nua06, FLS01}.

Let $\mathcal{H}$ be a real separable Hilbert space. If there exist a complete probability space $(\Omega,\mathcal{F},\mathbb{P})$ and a linear map $W : \mathcal{H} \to L^2(\Omega)$ such that $W(h)$ are centered Gaussian random variables with covariances given by $\mathbb{E}(W(h)W(k)) = \langle h,k \rangle_{\mathcal{H}}$, $W$ is called a Gaussian process on $\mathcal{H}$.

Now, let $\mathcal{F}$ be the $\sigma$-algebra generated by the random variables $\{ W(h) \mid h \in \mathcal{H} \}$.

We denote by $C^{\infty}_b$ the set of all infinitely continuously differentiable functions $ f : \mathbb{R}^n \to \mathbb{R} $ such that $f$ and all of its partial derivatives are bounded. The algebra $\mathcal{S}$ of smooth random variables is defined by
$$
\mathcal{S} = \{ F = f( W(h_1),\dots,W(h_n) ) \mid n \in \mathbb{N}, f \in C^{\infty}_b(\mathbb{R}^n), h_1, \dots, h_n \in \mathcal{H} \},
$$
and the derivation operator $\mathcal{D} : \mathcal{S} \to L^2(\Omega) \otimes \mathcal{H} \cong L^2 (\Omega;\mathcal{H})$ is defined by
$$
\mathcal{D}(F) = \sum^n_{i=1} \frac{\partial f}{\partial x_i} (W(h_1), \dots, W(h_n)) \otimes h_i
$$
for $F = f( W(h_1), \dots, W(h_n) ) \in \mathcal{S}$.
$\mathcal{D}$ has various properties, among which we list the following:
\begin{enumerate}
\item (The Leibniz rule) $\mathcal{D} (FG) = \mathcal{D}(F)G + F \mathcal{D}(G)$\ for $ F,G \in \mathcal{S} $;
\item (Integration by parts) $\mathbb{E} (FGW(h)) = \mathbb{E}( \langle h,\mathcal{D}(F)\rangle_{\mathcal{H}} G + F \langle h, \mathcal{D}(G) \rangle_{\mathcal{H}} )$\ for $ F, G \in \mathcal{S},\ h \in \mathcal{H} $;
\item $\mathcal{D}$ is closable from $ L^p(\Omega) $ to $ L^p(\Omega ; \mathcal{H}) $ for $ 1 \le p \le \infty $. The closure of $\mathcal{D}$ is also denoted by $\mathcal{D}$;
\item the domain of $\mathcal{D}$ in $L^p(\Omega)$, denoted by $\mathbb{D}^{1, p}$, is determined by the completion of $\mathcal{S}$ with respect to the norm $\| F \|_{1,p} = [ \mathbb{E}(| F |^p) + \mathbb{E}(\| \mathcal{D}(F) \|^p_{\mathcal{H}} ]^{1/p}$.
\end{enumerate}
The adjoint of $\mathcal{D} : L^2(\Omega) \to L^2(\Omega;\mathcal{H})$ can be defined. This is called the divergence operator, denoted by $\delta : L^2(\Omega;\mathcal{H}) \to L^2(\Omega)$. If we denote the smooth elementary elements by
$$
\mathcal{S}_{\mathcal{H}} \coloneqq \left\{ u = \sum^n_{j=1} F_j\otimes h_j \mid n \in \mathbb{N}, F_1, \dots, F_n \in \mathcal{S}, h_1, \dots, h_n \in \mathcal{H} \right\},
$$
then the following holds:
$$
\delta(u) = \sum^n_{j=1} F_j W(h_j) - \sum^n_{j=1} \langle h_j, \mathcal{D}(F_j) \rangle_{\mathcal{H}}
$$
for $u = \sum^n_{j=1} F_j \otimes h_j \in \mathcal{S}_{\mathcal{H}}$.
Regarding the properties of $\delta$, we list the following:
\begin{enumerate}
\item $\mathcal{D}_{h}(\delta(u)) = \langle h, u \rangle_{\mathcal{H}} + \delta( \mathcal{D}_h (u) ) $, where $ \mathcal{D}_h (F) = \langle h, \mathcal{D}(F) \rangle_{\mathcal{H}} $;
\item $\mathbb{E} ( \delta(u) \delta(v) ) = \mathbb{E}({\langle u,v \rangle_{\mathcal{H}}}) + \mathbb{E}(\mathrm{Tr}(\mathcal{D}(u) \circ \mathcal{D}(v))) $, where $\langle e_i \rangle_{i \in \mathbb{N}}$ is a complete orthonormal system of $\mathcal{H}$, and $\mathrm{Tr}((\mathcal{D}(u) \circ \mathcal{D}(v))) = \sum_{i, j} \mathcal{D}_{e_i}( \langle u, e_j \rangle_{\mathcal{H}}) \mathcal{D}_{e_j} (\langle  v, e_i \rangle_{\mathcal{H}}) $;
\item $ \delta(Fu) = F\delta(u) - \langle u, \mathcal{D}(F) \rangle_{\mathcal{H}} $ for $h \in \mathcal{H}, u, v \in \mathcal{S}_{\mathcal{H}}, F \in \mathcal{S}$.
\end{enumerate}
From here on, we assume $ \mathcal{H} = L^2(\mathbb{R}_+) $. This is the $L^2$ space over the set of non-negative real numbers $\mathbb{R}_+$, the Borel $\sigma$-algebra $\mathcal{B}$ on $\mathbb{R}_+$, and the Lebesgue measure $\mu(=ds)$. We denote by $L^2_{s} (\mathbb{R}^m_{+})$ the space of symmetric functions contained in $L^2(\mathbb{R}^m_+)$.
For $f \in L^2_s(\mathbb{R}^m_+)$, we denote the multiple stochastic integrals with respect to $W$ by $I_m(f)$~\cite{Nua06}. To explain briefly, for $m\in \mathbb{N}$, let $\mathcal{B}_0 = \{ A \in \mathcal{B} \mid \mu(A) < \infty \} $, and let $A_1, \dots, A_m$ be pairwise disjoint sets belonging to $\mathcal{B}_0$. We define
$$
I_m( \tilde{\mathbbm{1}}_{A_1 \times \dots \times A_m}) = W(\mathbbm{1}_{A_1})\cdots W(\mathbbm{1}_{A_m}),
$$
and extend it using linearity and $L^2$ isometry to get $I_m(f)$. Here, $\tilde{f}$ is the symmetrization of $f$, defined by
$$
\tilde{f}(t_1, \dots, t_m) = \frac{1}{m!} \sum_{\sigma \in S_m} f(t_{\sigma(1)}, \dots, t_{\sigma(m)}),
$$
where $S_m$ is the $m$-th symmetric group. This satisfies
$$
\mathbb{E}( I_m(f) I_q(g) ) = \left\{
\begin{array}{ll}
0 & (m \neq q)\\
m!\langle f, g \rangle_{L^2(\mathbb{R}^m_+)} & ( m = q ).
\end{array}
\right.
$$
We also write $\int f(t_1, \dots, t_m)\ W(dt_1)\cdots W(dt_m)$ for $I_m(f).$
We now list some properties of $ I_m(f) $.
\begin{prop}[{\cite[Prop.~1.1.3]{Nua06}}]\label{multiple formula}
    For every $p, q \in \mathbb{Z}_{\ge 0}$, let $f \in L^2_s(\mathbb{R}^p_+),\ g \in L^2_s(\mathbb{R}^q_+)$. Then
    $$
    I_p(f)I_q(g) = \sum^{p\wedge q}_{r=0} r! \binom{p}{r} \binom{q}{r} I_{p+q-2r}(f \tilde{\otimes}_r g),
    $$
    where  $f \tilde{\otimes}_r g$ is the symmetrization of $$f {\otimes}_r g \coloneqq \int f(t_1, \dots, t_{p-r}, s_1, \dots, s_r ) g( t_{p-r+1}, \dots, t_{p+q-2r}, s_1, \dots, s_{r} )\ d\bf{s} .$$
\end{prop}
\begin{prop}[{\cite[Prop.~1.1.4]{Nua06}}]\label{hermite}
    Let $H_m(x)$ be the $m$-th Hermite polynomial, and let $h \in L^2(\mathbb{R}_+)$ be an element of norm one. Then
    $$
    H_m(W(h)) = \int h(t_1) \cdots h(t_m) \ W(dt_1) \cdots W(dt_m).
    $$
\end{prop}
Here, $H_m$ is defined as $H_m(x) = (-1)^{m} \exp{(x^2/2)} d^m/dx^m \exp{(-x^2/2)} $. Additionally, by considering the parameterized Hermite polynomial
$$
H_m(x; \sigma^2) = (-\sigma^2)^m \exp\left({\frac{x^2}{2\sigma^2}}\right) \frac{d^m}{dx^m} \exp\left(-{\frac{x^2}{2\sigma^2}}\right),
$$
the norm of $h$ in the proposition does not need to be $1$. The proposition extends in the following way:
$$
H_m(W(h);\| h \|^2_{L^2(\mathbb{R}_+)}) = \int h(t_1) \cdots h(t_m) \ W(dt_1) \cdots W(dt_m).
$$
For $ H_m(x; \sigma^2) $, note the recurrence relations:
\begin{gather*}
    H_0( x; \sigma^2 ) = 1,\\
    H_1( x; \sigma^2 ) = x,\\
    xH_n( x; \sigma^2 ) = H_{n+1}( x; \sigma^2 ) + \sigma^2 H_{n-1}( x; \sigma^2 ).
\end{gather*}
\begin{thm}[{\cite[Thm~1.1.4]{Nua06}} the Wiener chaos expansion]
Any square integrable random variable $F \in L^2(\Omega)$ can be expanded into a series of multiple stochastic integrals:
$$
F= \sum^{\infty}_{n=0} I_n(f_n).
$$
Here $f_0 = \mathbb{E}(F)$, and $ f_n \in L^2_s(\mathbb{R}^{n}_+)$ are uniquely determined by $F$.
\end{thm}
We also introduce some properties of $\mathcal{D}$ and $ \delta$ when $ \mathcal{H} = L^2(\mathbb{R}_+) $. For $F\in \mathbb{D}^{1, 2}$, let $\{\mathcal{D}_t(F)\mid t\in\mathbb{R}_+\}$ be a stochastic process given by $\mathcal{D}(F)$ due to the identification between the Hilbert spaces $L^2(\Omega;\mathcal{H})$ and $L^2(\mathbb{R}_+\times \Omega)$.
\begin{prop}[{\cite[Propositions~1.2.2 and 1.2.7]{Nua06}}]
Suppose that $F\in L^2(\Omega)$ has the chaos expansion as $F=\sum^{\infty}_{n=0}I_n(f_n)$. Then $F \in \mathbb{D}^{1,2}$ if and only if 
$$
\sum^\infty_{n=1} n \| I_n(f_n) \|^2_{L^2(\Omega)} = \sum^\infty_{n=1} n \cdot n! \| f_n \|^2_{L^2(\mathbb{R}^n_+)} < \infty.
$$
If $F \in \mathbb{D}^{1,2}$, then
$$
\mathcal{D}_t(F) =\sum^\infty_{n=1}n I_{n-1}(f_n(\cdot,t))
$$
for a.e. $t\ge 0$.
\end{prop}
\begin{prop}[{\cite[Prop.~1.3.7]{Nua06}}]
    $u\in L^2(\mathbb{R}_+ \times \Omega)$ with the expansion $u(t) =\sum^\infty_{n=0}I_n(f_n(\cdot,t)).$ Then $u \in \mathrm{Dom}(\delta)$ if and only if 
    $$
    \delta(u) = \sum^\infty_{n=0} I_{n+1}(\Tilde{f_n})
    $$
    converges in $L^2(\Omega).$
\end{prop}
For the next two propositions, we assume $ \mathcal{H} = L^2([0,1]) $ and let $W = \{ W(t) \mid 0 \le t \le 1\}$ represent the $1$-dimensional Brownian motion. The following proposition describes the relationship between the It\^o stochastic integral and the divergence operator. Let $A \in\mathcal{B}$. We define $\mathcal{F}_A$ by the $\sigma$-field generated by $\{ W(B) \mid B\subset A, B \in \mathcal{B}_0 \}$, and let $\mathcal{F}_t = \mathcal{F}_{[0,t]}$. We denote by $L^2_a$ the closed subspace of $L^2([0, 1]\times \Omega)$ formed by the $\{ \mathcal{F}_t \}$-adapted processes.
\begin{prop}[{\cite[Prop.~1.3.11]{Nua06}}]
    $L^2_a \subset \mathrm{Dom}(\delta)$, and the operator $\delta$ restricted to $L^2_a$ coincides with the It\^o integral:
    $$
    \delta(u) = \int^1_0 u_t\ dW_t
    $$
    for $u \in L^2_a $.
\end{prop}
The following formula is called the Clark-Ocone formula.
\begin{prop}[{\cite[Prop.~1.3.14]{Nua06}}]
    Let $F \in \mathbb{D}^{1,2} $. It holds that
    $$
    F = \mathbb{E}(F) + \int^1_0 \mathbb{E}(\mathcal{D}_t(F)|\mathcal{F}_t)\ dW_t.
    $$
\end{prop}
So far, we have briefly reviewed Malliavin calculus, and it is important to note that $\mathcal{D}_t$ and $ \delta$ satisfy the canonical commutation relations (CCR) $[\mathcal{D}_t, \delta] = \delta_t \otimes \mathrm{id}_{L^2(\Omega)} $, where $\delta_t$ is the Dirac delta function concentrated at $t$. This suggests that there might be a hidden non-commutative structure in Brownian motion and Malliavin calculus. In the next subsection, we will introduce the relationship between the creation and annihilation operators of a boson, and define Brownian motion on the boson.

\subsection{Relationship between Brownian Motion and a Boson}\label{subsection_BB_relation}
The content, terminology, and notation in this subsection are based on \cite{Par92, Mey95, FLS01}. As previewed at the end of subsection~\ref{subsection_Malliavin}, we will introduce a non-commutative probability space corresponding to Brownian motion and list some of its properties.
Let $\mathcal{H}$ be a real separable Hilbert space, and let $\mathcal{H}_{\mathbb{C}}$ be its complexification. The conjugation $J : \mathcal{H}_{\mathbb{C}} \to \mathcal{H}_{\mathbb{C}} $ is defined by 
$$
J(h_1 + \sqrt{-1}h_2) = h_1 - \sqrt{-1}h_2
$$
for $h_1, h_2 \in \mathcal{H}$.
The inner product on $\mathcal{H}_{\mathbb{C}}$ is given by 
\begin{align*}
\langle h_1+ \sqrt{-1}h_2 , k_1 +  \sqrt{-1}k_2\rangle_{\mathcal{H}_\mathbb{C}} &\coloneq \langle h_1, k_1 \rangle_{\mathcal{H}} + \langle h_2, k_2 \rangle_{\mathcal{H}} \\
&\quad+ \sqrt{-1} \langle h_1, k_2 \rangle_{\mathcal{H}} - \sqrt{-1} \langle h_2, k_1 \rangle_{\mathcal{H}}
\end{align*}
for $h_1, h_2, k_1, k_2 \in \mathcal{H}$.
An element $h$ is said to be real if $J(h) = h$. Let $n\in\mathbb{N}$. $\odot$ represents the symmetric tensor product:
$$
\xi_1 \odot \cdots \odot \xi_n = \frac{1}{n!} \sum_{\sigma\in S_n} \xi_{\sigma(1)} \otimes \cdots \otimes \xi_{\sigma(n)} 
$$
for $\xi_1,\dots,\xi_n \in \mathcal{H}_\mathbb{C}$.
We define $\mathcal{H}_{\mathbb{C}}^{\odot n}$ to be the closed subspace of $\mathcal{H}_{\mathbb{C}}^{\otimes n}$ generated by the set $\{ \xi_1 \odot \cdots \odot \xi_n \mid \xi_1,\dots,\xi_n \in \mathcal{H}_\mathbb{C} \}$.
The inner product on $\mathcal{H}_{\mathbb{C}}^{\odot n}$ is inherited from the inner product on $\mathcal{H}_{\mathbb{C}}^{\otimes n}$.
For $f \in \mathcal{H}_{\mathbb{C}}^{\odot n}$ and $ g\in \mathcal{H}_{\mathbb{C}}^{\odot m}$, the notation $f \odot g$ is also used to denote the symmetrization of $f\otimes g$.
Let $\Gamma_s = \bigoplus^\infty_{n=0} \mathcal{H}_{\mathbb{C}}^{\odot n} $ denote the symmetric Fock space over $\mathcal{H}_{\mathbb{C}}$, where $\mathcal{H}_{\mathbb{C}}^{\odot 0}=\mathbb{C}$. The vacuum vector is denoted by $\Omega \in \Gamma_s$. The set $\{ \xi^{\otimes n} \mid \xi \in \mathcal{H}_{\mathbb{C}}, n\in \mathbb{N} \}$ is total in $\Gamma_s $. The creation and annihilation operators $a^*(h), a(h)$, where $h\in \mathcal{H}_{\mathbb{C}}$, are closed and unbounded operators that are adjoint to each other. They act on $\xi^{\otimes n}$ as follows:
\begin{gather*}
    a^*(h)\xi^{\otimes n}=\sqrt{n}\langle h, \xi \rangle_{\mathcal{H}_{\mathbb{C}}}\xi^{\otimes n-1},\\
    a(h)\xi^{\otimes n}=\sqrt{n+1} h \odot \xi^{\otimes n}.
\end{gather*}
Furthermore, the operators $a^*$ and $a$ satisfy CCR:
\begin{gather*}
    [a(h), a^*(k)]= \langle h, k \rangle_{\mathcal{H}_{\mathbb{C}}},\\
    [a(h), a(k)]= [a^*(h), a^*(k)]=0
\end{gather*}
for $h, k \in \mathcal{H}_{\mathbb{C}}.$ We define $B(h) \coloneq a^*(h) + a(J(h))$, where $h\in\mathcal{H}_{\mathbb{C}}.$
\begin{thm}
    For every $m\in\mathbb{N}$, let $f_1, \dots, f_{2m} \in \mathcal{H}_{\mathbb{C}}$. Then
    \begin{gather*}
        \langle \Omega, B(f_1)\dots B(f_{2m-1})\Omega \rangle_{\Gamma_s}=0,\\
        \langle \Omega, B(f_1)\dots B(f_{2m})\Omega \rangle_{\Gamma_s}= \sum \langle J(f_{l_1}), f_{r_1} \rangle_{\mathcal{H}_{\mathbb{C}}}\cdots \langle J(f_{l_m}), f_{r_m} \rangle_{\mathcal{H}_{\mathbb{C}}} ,
    \end{gather*}
    where the sum is taken over all possible ways of partitioning $\{ 1, 2, \dots, 2m \}$ into $m$ distinct pairs $(l_1, r_1), \dots, (l_m, r_m),\ l_k < r_k$ (there are $(2m)!/2^m m!$ such partitions).
\end{thm}
This result can be proved by iteratively using the CCR (cf. Wick's theorem \cite{wick1950evaluation}). From now on, let $\mathcal{H}=L^2(\mathbb{R}_+),\ a_t = a(\mathbbm{1}_{[0, t]}),\ a^*_t = a^*(\mathbbm{1}_{[0, t]})$, and $ B_t = a_t + a^*_t$.
\begin{cor}
    The moments of $B_t$ are equal to those of a standard Brownian motion.
\end{cor}
Now, as we observed in Proposition~\ref{hermite}, the multiple integrals of Brownian motion are represented by Hermite polynomials. The following theorem describes how this can be expressed from the perspective of non-commutative probability theory.
\begin{thm}
    $$
    H_m(B_t)= :B_t^n: =\int_0^t \cdots\int_0^t \ dB_{t_1} \cdots dB_{t_n},
    $$
    where $: \cdot :$ denotes the Wick-ordered product. The integral is interpreted as the iterated integral in the sense of Hudson-Parthasarathy \cite{Par92,HS81}.
\end{thm}
\begin{proof}
First, we will prove the first equality. It suffices to show the recurrence relation $B_t :B_t^n: =:B_t^{n+1}: + nt:B_t^{n-1}:$. For $n=1$,
\begin{gather*}
    B_t \cdot B_t= a^*_t a^*_t + a^*_t a_t + a_t a^*_t +a_t a_t,\\
    : B_t^2 : = a^*_t a^*_t +2 a^*_t a_t + a_t a_t.
\end{gather*}
Therefore,
$$
B_t :B_t: = :B_t^2: + 1\cdot t :B_t^0:.
$$
Assuming the recurrence holds up to $n=k-1$, we have
\begin{align*}
    B_t :B_t^{k-1}:&= :B_t^k: + kt:B_t^{k-2}:,\\
        B_t :B_t^{k-1}: - :B_t^k: &= (a^*_t :B_t^{k-1}: + a_t :B_t^{k-1}:) - (a^*_t :B_t^{k-1}: + :B_t^{k-1}:a_t)\\
        &=a_t :B_t^{k-1}: - :B_t^{k-1}:a_t
\end{align*}
for $n=k\ge 2$.
Thus,
$$
    a_t :B_t^{k-1}:-:B_t^{k-1}:a_t=kt:B_t^{k-2}:.
$$
Consequently,
 \begin{align*}
        B_t :B_t^k: - :B_t^{k+1}: &= a_t:B_t^k:-:B_t^k:a_t\\
        &= a_t a^*_t :B_t^{k-1}:+a_t:B_t^{k-1}:a_t -:B_t^k:a_t\\
        &= a^*_t a_t :B_t^{k-1}: + t :B_t^{k-1}: + :B_t^{k-1}:a_ta_t + (k-1)t:B_t^{k-2}:a_t \\
        &\quad -a^*_t:B_t^{k-1}:a_t-:B_t^{k-1}:a_ta_t\\
        &= (k-1)t:B_t^{k-2}:a_t+a^*_t(k-1)t:B_t^{k-2}:+t:B_t^{k-1}:\\
        &= (k-1)t:B_t^{k-1}:+t:B_t^{k-1}:\\
        &= kt:B_t^{k-1}:.
    \end{align*}
Next, we will prove the second equality. Let $f, g\in\mathcal{H}_{\mathbb{C}}$. Consider the exponential vectors
$$
\mathcal{E}(f) = \sum^\infty_{n=0} \frac{1}{n!} f^{\otimes n}
$$
and $\mathcal{E}(g)$. The set $\{ \mathcal{E}(f) \mid f \in \mathcal{H}_{\mathbb{C}} \}$ is also total in $\Gamma_s $. We have
\begin{align*}
        &\langle \mathcal{E}(g), \int_0^t \cdots\int_0^t \ dB_{t_1} \cdots dB_{t_n} \mathcal{E}(f) \rangle_{\Gamma_s}\\
        &= \langle \mathcal{E}(g), \int_0^t \left( \int_0^t \cdots\int_0^t \ dB_{t_1} \cdots dB_{t_{n-1}} \right) dB_{t_n} \mathcal{E}(f) \rangle_{\Gamma_s}\\
        &= \int_0^t \langle \mathcal{E}(g), \int_0^t \cdots\int_0^t \ dB_{t_1} \cdots dB_{t_{n-1}} f(t_n)\mathcal{E}(f) \rangle_{\Gamma_s}\\
        &\quad+ \langle g(t_n)\mathcal{E}(g), \int_0^t \cdots\int_0^t \ dB_{t_1} \cdots dB_{t_{n-1}} \mathcal{E}(f) \rangle_{\Gamma_s}\ dt_n\\
        &= \langle \mathcal{E}(g), \left(\int_0^t \cdots\int_0^t \ dB_{t_1} \cdots dB_{t_{n-1}}a_t + a^*_t\int_0^t \cdots\int_0^t \ dB_{t_1} \cdots dB_{t_{n-1}} \right)\mathcal{E}(f) \rangle_{\Gamma_s}.
\end{align*}
Using a density argument, by induction, we have the result:
$$
    \int_0^t \cdots\int_0^t \ dB_{t_1} \cdots dB_{t_{n}} = :B_t^n:.
    $$
\end{proof}
Up to this point, we have confirmed that Malliavin calculus, particularly the structure behind the chaos expansion, is intricately linked with non-commutative probability theory and the creation and annihilation operators of a boson. This observation leads to two questions:
\begin{itemize}
    \item Is there a noncommutative space, possessing an algebraic structure similar to (but distinct from) that of Brownian motion, which also admits the chaos expansion?
    \item If such a space exists, can we construct Malliavin calculus within the space?
\end{itemize}
The first question has been addressed by \cite{BSW1}, which shows that the Clifford algebra satisfies these conditions. Section~\ref{subsection_Mcf} will give the affirmative resolution of the second question. Additionally, for prior research on the derivation operator on the Clifford algebra, see \cite{AMN21}.
% 3
\section{Stochastic Calculus on the Clifford Algebra}\label{section_CAR algebra}
In this section, we will introduce the necessary contents derived from \cite{BSW1}.
Again, let $\mathcal{H}_{\mathbb{C}}$ be the complexification of a general real separable Hilbert space $\mathcal{H}$. Let $n\in\mathbb{N}$.
$\wedge$ represents the anti-symmetric tensor product, defined by 
$$
z_1 \wedge \cdots \wedge z_n = \frac{1}{n!} \sum_{\sigma\in S_n} \mathrm{sgn}(\sigma) z_{\sigma(1)} \otimes \cdots \otimes z_{\sigma(n)}
$$
for $z_1,\dots,z_n \in\mathcal{H}_{\mathbb{C}}$.
We define $\Lambda_n(\mathcal{H}_{\mathbb{C}})$ to be the closed subspace of $\mathcal{H}_{\mathbb{C}}^{\otimes n}$ generated by the set $\{ z_1 \wedge \cdots \wedge z_n \mid z_1,\dots,z_n \in \mathcal{H}_\mathbb{C} \}$.
Note that the inner product on $\Lambda_n(\mathcal{H}_{\mathbb{C}})$ is given by 
$$
\langle  z_1\wedge \cdots \wedge z_n , 
 z'_1 \wedge \cdots \wedge z'_n\rangle_{\Lambda_n(\mathcal{H}_{\mathbb{C}})}=\mathrm{det}\langle z_i, z'_j\rangle_{\mathcal{H}_{\mathbb{C}}},
$$
rather than $1/n\cdot\mathrm{det}\langle z_i, z'_j \rangle_{\mathcal{H}_{\mathbb{C}}}$.
For $f \in \Lambda_n(\mathcal{H}_{\mathbb{C}})$ and $ g\in \Lambda_m(\mathcal{H}_{\mathbb{C}})$, $f \wedge g$ is also used to denote the anti-symmetrization of $f\otimes g$. 
Let $\Gamma_a = \bigoplus^\infty_{n=0} \Lambda_n(\mathcal{H}_{\mathbb{C}}) $ denote the anti-symmetric Fock space over $\mathcal{H}_{\mathbb{C}}$, where $\Lambda_0(\mathcal{H}_{\mathbb{C}})=\mathbb{C}$.
The creation and annihilation operators $b^*(z), b(z)$ are bounded and mutually adjoint operators, whose operator norms satisfy $\| b(z) \| =\| b^*(z) \| =\|z\|_{\mathcal{H}_{\mathbb{C}}}$. For $ w = z_1 \wedge \cdots \wedge z_n$, the action of these operators is given by:
\begin{gather*}
    b^*(z) w =  z \wedge w,\\
    b(z)z_1 \wedge \cdots \wedge z_n=\sum_{i=1}^n (-1)^{i-1} \langle z, z_i \rangle_{\mathcal{H}_{\mathbb{C}}} z_1 \wedge \overset{i}{\Check{\cdots}} \wedge z_n.
\end{gather*}
The operators $b^*, b$ satisfy the canonical anti-commutation relations (CAR):
\begin{align*}
    \{b(z), b^*(z')\}&= \langle z, z' \rangle_{\mathcal{H}_{\mathbb{C}}},\\
    \{b(z), b(z')\}&= \{b^*(z), b^*(z')\}=0
\end{align*}
for $z, z' \in \mathcal{H}_{\mathbb{C}}$. We define $\Psi(z) \coloneq b^*(z) + b(J(z))$. The operator $\Psi$ also satisfies the CAR:
$$
\{ \Psi(z), \Psi(z') \} = \Psi(z)\Psi(z')+\Psi(z')\Psi(z)=2\langle J(z'), z \rangle_{\mathcal{H}_{\mathbb{C}}}
$$
for $z, z' \in \mathcal{H}_{\mathbb{C}}.$ 
Let $\mathscr{C}$ be the $W^*$-algebra generated by $\{ \Psi(z) | z\in\mathcal{H}_{\mathbb{C}} \}$. Let $\Omega \in \Gamma_a$ denote the vacuum vector. We define $m(\cdot)=\langle \Omega, \cdot \Omega \rangle_{\Gamma_a}$. This becomes a faithful and central state, making $(\Gamma_a, \mathscr{C}, m)$ a regular probability gauge space \cite{Seg53,Seg56b,Gro72}.
For $1\le p < \infty$, $L^p(\mathscr{C})$ is defined as the completion of $\mathscr{C}$ with respect to the norm  $\| u \|_{L^p(\mathscr{C})} =m(|u|^p)^{\frac{1}{p}}=\langle \Omega, (u^*u)^{\frac{p}{2}}\Omega \rangle^{\frac{1}{p}}_{\Gamma_a}$. $L^\infty(\mathscr{C})$ represents $\mathscr{C}$ equipped with the operator norm. Let $\mathscr{B}$ be a $W^*$-subalgebra of $\mathscr{C}$. Regarding the properties of $L^p(\mathscr{C})$, we list the following:
\begin{enumerate}
    \item $L^p(\mathscr{C}) \supset L^q(\mathscr{C})$ and $\| \cdot \|_{L^p(\mathscr{C})} \le \| \cdot \|_{L^q(\mathscr{C})}$ for $1\le p \le q \le \infty$;
    \item $L^{p'}(\mathscr{C})$ is the dual of $L^p(\mathscr{C})$ for $1\le p < \infty,\ p'= \frac{p}{p-1}\ \cite{Seg53,Kun58,Yea75}$;
    \item $L^p(\mathscr{B})$ is a closed subspace of $L^p(\mathscr{C}).$
\end{enumerate}
For any $1\le p\le \infty$, the conditional expectation $m(u|\mathscr{B})$ of $u\in L^p(\mathscr{C})$ with respect to $\mathscr{B}$ is defined. $m(\cdot|\mathscr{B})$ satisfies the following properties \cite{Seg53,Kun58,Yea75,Wil74,Ume54,Ume56}:
\begin{enumerate}
    \item $m\big( m(u|\mathscr{B}) v \big) =m(uv)$ for $u\in L^p(\mathscr{C}), v\in L^{p'}(\mathscr{B})$ with $p' = \frac{p}{p-1}$;
    \item $m(\cdot | \mathscr{B})$ is a contraction from $L^p(\mathscr{C})$ onto $L^p(\mathscr{B})$ for all $1 \le p \le \infty$;
    \item $m(\cdot | \mathscr{B})$ is positivity preserving;
    \item $m(vu|\mathscr{B})=v\cdot m(u|\mathscr{B})$ for all $u\in L^p(\mathscr{C}), v\in L^{p'}(\mathscr{B})$;
    \item if $\mathscr{B}_1 \subset \mathscr{B}_2,$ then $m\big( m(\cdot |\mathscr{B}_2) | \mathscr{B}_1 \big)= m(\cdot | \mathscr{B}_1).$
\end{enumerate}
The map $u \mapsto u\Omega$ from $\mathscr{C}$ into $\Gamma_a$ extends to a unitary operator $D:L^2(\mathscr{C}) \to \Gamma_a$\ \cite{Seg56b,Gro72}.
Let $\mathcal{H}=L^2(\mathbb{R}_+)$. The complexification of $L^2(\mathbb{R}_+)$ is denoted by $L^2(\mathbb{R}_+)$ as well. We use $\Lambda_n$ for $\Lambda_n(L^2(\mathbb{R}_+))$, which consists of all anti-symmetric elements of $L^2(\mathbb{R}_+^n)$. Note that the inner product on $\Lambda_n$ is given by 
$$
\langle f, g \rangle_{\Lambda_n} = n! \int \overline{f(x_1, \dots,x_n)} g(x_1, \dots,x_n)\ d{\bf x}.
$$
Let $\mathscr{C}_t$ be the $W^*$-subalgebra of $\mathscr{C}$ generated by $\{ \Psi(u) \mid u\in L^2(\mathbb{R}_+), \supp(u)\subset [0, t] \}$.
\begin{dfn}[{\cite[Def.~2.1]{BSW1}}]
    For $1\le p\le \infty$, an $L^p$-martingale adapted to the family $\{ \mathscr{C}_t\}$ is a collection $\{X_t \mid t \in \mathbb{R}_+ \}$ of operators on $\Gamma_a$ such that $X_t \in L^p(\mathscr{C}_t)$ for $t\in \mathbb{R}_+,$ and $m( X_t | \mathscr{C}_s) = X_s$ for any $0 \le s \le t.$
\end{dfn}
Let $u_1, \dots, u_n \in L^2_{loc}(\mathbb{R}_+), :\Psi(u_1\mathbbm{1}_{[0, t]})\cdots \Psi(u_n\mathbbm{1}_{[0, t]}) :$ denotes the Wick-ordered product.
\begin{thm}[{\cite[Thm~2.2]{BSW1}}]
    $\{ :\Psi(u_1\mathbbm{1}_{[0, t]})\cdots \Psi(u_n\mathbbm{1}_{[0, t]}) : \mid t \in \mathbb{R}_+ \}$ is an $L^\infty$-martingale adapted to the family $\{ \mathscr{C}_t \}.$
\end{thm}
$\int w_n (x_1, \dots, x_n) :\psi(x_1) \cdots \psi(x_n): dx_1 \cdots dx_n$ denotes the element in $L^2(\mathscr{C})$ such that 
$$
\int w_n (x_1, \dots, x_n) :\psi(x_1) \cdots \psi(x_n): dx_1 \cdots dx_n \Omega = w_n 
$$
for $w_n \in \Lambda_n$\ \cite{Hep69}. We abbreviate it as  $\int w_n :\psi: d{\bf x},\ J_n(w_n(x_1, \dots, x_n))$, and $ J_n(w_n)$. Note that $J_n(w_n(x_1, \dots, x_n))^* = J_n(\overline{w_n(x_n, \dots, x_1)}),\ m(\int w_n :\psi: d{\bf x} | \mathscr{C}_t) = \int w_n \mathbbm{1}_{[0, t]^n} :\psi: d{\bf x}$.
\begin{thm}[{\cite[Thm~2.4]{BSW1}}]
    Let $\{ X_t \mid t\in \mathbb{R}_+ \}$ be an $L^2$-martingale adapted to the family $\{ \mathscr{C}_t \}.$ Then there exist $w_0 \in \mathbb{C}$ and the anti-symmetric functions $w_n \in L^2_{loc}(\mathbb{R}^n_+)$ such that 
    $$
    X_t = w_0 + \sum^\infty_{n=1} \int w_n \mathbbm{1}_{[0, t]^n} :\psi: d{\bf x} .
    $$
\end{thm}
Here, we introduce the It\^o-Clifford integral, which is an It\^o integral associated with the fermion field.
\begin{dfn}[{\cite[Def.~3.1]{BSW1}}]
    Let $0 \le t_0 \le t.$ An $L^p$-adapted process on $[t_0, t]$ is a map $f: [t_0, t] \to L^p(\mathscr{C})$ such that $f(s) \in L^p(\mathscr{C}_s) $ for all $ s \in [t_0, t].$
\end{dfn}
\begin{dfn}[{\cite[Def.~3.2]{BSW1}}]
    An $L^1$-adapted process $h$ on $[t_0,t]$ is said to be simple if it can be expressed as
    $$
    h = \sum_{k=1}^n h_{k-1} \mathbbm{1}_{[ t_{k-1}, t_k )}
    $$
    on $[ t_0, t )$, for some $t_0 \le t_1 \le \cdots \le t_n = t$ and $h_k \in L^1(\mathscr{C}_{t_k})$, $0\le k \le n-1.$
\end{dfn}
By an abuse of terminology, we will also consider a function $f \in L^2([t_0, t];L^2(\mathscr{C}))$ as an adapted process if $f(s)\in L^2(\mathscr{C}_s)$ for $\mathrm{a.e.}\ s\in [t_0, t]$. We write $\mathfrak{H}[t_0, t]$ for the set of these adapted processes in $L^2([t_0, t];L^2(\mathscr{C}))$.
\begin{thm}[{\cite[Thm~3.10, Def.~3.11]{BSW1}}]
    For $f \in \mathfrak{H}[t_0, t],$ the It\^o-Clifford integral of $f$ exists.
\end{thm}
Let $\int^t_{t_0} d\Psi_s f(s)$ denote the It\^o-Clifford integral, where $\Psi_s = \Psi(\mathbbm{1}_{[0,s]})$.
This is defined through the following procedure. For a simple adapted process, we define the integral as
$$
\int^t_{t_0} d\Psi_s \big( \sum_{k=1}^n h_{k-1} \mathbbm{1}_{[ t_{k-1}, t_k )} \big) \coloneq \sum _{k=1}^n (\Psi_{t_k} - \Psi_{t_{k-1}})h_{k-1}.
$$
Then the integral for more general processes is defined as the $L^2(\mathscr{C})$-limit of this expression. Note that, in contrast to the original paper, we consider the left integral.
At the end of this section, we introduce the conjugation operator $\beta$ on $L^2(\mathscr{C})$. 
\begin{dfn}[{\cite[Def.~3.13]{BSW1}}]
    $\beta: L^2(\mathscr{C}) \to L^2(\mathscr{C})$ denotes the self-adjoint unitary $D^{-1}\Gamma(-1) D,$ where $\Gamma(-1)$ is the second quantization of $-1.$
\end{dfn}
For details on the properties of $\beta$, see \cite{BSW4}.

% 4
\section{Main Results}\label{section_Main}
% 4.1
\subsection{Malliavin Calculus on the Clifford Algebra}\label{subsection_Mcf}
In this subsection, we will develop Malliavin calculus for the Clifford algebra. The result corresponding to Proposition~\ref{multiple formula} will form a crucial foundation for the subsequent developments.
\begin{prop}\label{anti-multiple formula}
Let $p, q \in \mathbb{Z}_{\ge 0}$. For every $f \in \Lambda_p,\ g \in \Lambda_q$, the following formula holds:
\begin{align}
    J_p (f) J_q (g) = \sum^{ p \wedge q }_{ r = 0 } r! \binom{p}{r} \binom{q}{r} J_{p+q-2r} ( f \hat{\wedge}_r g ), \label{anti-multiple eq.}
\end{align}
    where  $p \wedge q = \mathrm{min}(p,q)$, and $f \hat{\wedge}_r g$ denotes the anti-symmetrization of 
\begin{align}\label{arrow}
f {\wedge}_r g \coloneqq \int f(t_1, \dots, t_{p-r}, s_r, \dots, s_1 ) g(s_1, \dots, s_{r}, t_{p-r+1}, \dots, t_{p+q-2r} )\ d\bf{s}. 
\end{align}
\end{prop}
Before proceeding with the proof, we introduce some simplified notation.
 We use the notation
$$
    \int f(t_1, \dots, t_{p-r}, \overrightarrow{s_r} ) g(\overleftarrow{s_r}, t_{p-r+1}, \dots, t_{p+q-2r} )\ d\bf{s}
$$
for the right-hand side of (\ref{arrow}) and define
$$
    \overleftarrow{f}(x_1, \dots, x_p) \coloneq f(x_p, \dots, x_1).
$$
By using this notation, it is possible to write expressions such as $J_p(f)^* = J_p(\overline{\overleftarrow{f}})$.
\begin{proof}[Proof of Proposition~\ref{anti-multiple formula}]
First, we consider the case when $q=1$. Specifically, we aim to show that $J_p(f)J_1(g)=J_{p+1}(f \wedge g) + pJ_{p-1}(f\wedge_1 g)$. By using the linearity of $J_p$ and the fact that $\|J_p(f)\|_{L^2(\mathscr{C})}= \|f\|_{\Lambda_p}$, 
and by a density argument, it suffices to consider the functions
$$
    f = \hat{\mathbbm{1}}_{A_1\times \dots \times A_p},\  g = \mathbbm{1}_{A_0} \ \text{or $\mathbbm{1}_{A_p}$},
    $$
where $A_0, \dots, A_p$ are pairwise disjoint measurable sets.
When $g=\mathbbm{1}_{A_0}$, we have $f\wedge_1 g =0$, thus 
$$
    J_p(f)J_1(g) = \Psi(\mathbbm{1}_{A_1})\cdots\Psi(\mathbbm{1}_{A_p}) \cdot \Psi(\mathbbm{1}_{A_0}) = J_{p+1}(f\wedge g).
    $$
When $g=\mathbbm{1}_{A_p}$,
\begin{align}
        J_p(f)J_1(g) &= \Psi(\mathbbm{1}_{A_1})\cdots\Psi(\mathbbm{1}_{A_p}) \cdot \Psi(\mathbbm{1}_{A_p}) \nonumber\\
        &= \mu(A_p)\Psi(\mathbbm{1}_{A_1})\cdots\Psi(\mathbbm{1}_{A_{p-1}})\quad\text{($\mu$ is Lebesgue measure)} \nonumber\\
        &= \mu(A_p) J_{p-1}(\hat{\mathbbm{1}}_{A_1 \times \cdots \times A_{p-1}}) \nonumber\\
        &= \mu(A_p) \int \sum_{\sigma\in S_{p-1}} \frac{\mathrm{sgn}(\sigma)}{(p-1)!}\mathbbm{1}_{A_1\times \cdots \times A_{p-1}}(x_{\sigma(1)}, \dots, x_{\sigma(p-1)}) :\psi: d{\bf x}, \nonumber\\
        f\wedge_1 g &= \int \sum_{\sigma\in S_{p}} \frac{\mathrm{sgn}(\sigma)}{p!}\mathbbm{1}_{A_1\times \cdots \times A_{p}}(x_{\sigma(1)}, \dots, x_{\sigma(p)}) \mathbbm{1}_{A_p}(x_p) \ dx_p \nonumber\\
        &=\sum_{\sigma\in S_{p-1}} \int \frac{\mathrm{sgn}(\sigma)}{p!}\mathbbm{1}_{A_1\times \cdots \times A_{p}}(x_{\sigma(1)}, \dots, x_{\sigma(p-1)}, x_p) \mathbbm{1}_{A_p}(x_p) \ dx_p \label{4.1}\\
        &= \frac{\mu(A_p)}{p}\sum_{\sigma\in S_{p-1}} \frac{\mathrm{sgn}(\sigma)}{(p-1)!}\mathbbm{1}_{A_1\times \cdots \times A_{p-1}}(x_{\sigma(1)}, \dots, x_{\sigma(p-1)}) \nonumber\\
        &=  \frac{\mu(A_p)}{p} \hat{\mathbbm{1}}_{A_1\times \cdots \times A_{p-1}}, \nonumber
    \end{align}
where we used the fact that the integral is non-zero only when $\sigma(p)=p$ for (\ref{4.1}). Since $f\wedge g =0$, the proof for the case $q=1$ is completed.
For the general case of $q$, we use induction. Without loss of generality, let $p\ge q$. Assume that (\ref{anti-multiple eq.}) holds up to $q-1$. Using a density argument, as in the initial part, we can write $g$ as $g=g_1\wedge g_2$, where $g_1 \in \Lambda_{q-1},\ g_2 \in \Lambda_1$ such that $g_1\wedge g_2 = 0$. In this case, $J_q(g_1\wedge g_2) = J_{q-1}(g_1)J_1(g_2)$. Therefore,
    \begin{align*}
        J_p(f)J_q(g) &= J_p(f)J_{q-1}(g_1)J_1(g_2)\\
        &= \sum_{r=0}^{q-1}r! \binom{p}{r} \binom{q-1}{r}J_{p+q-1-2r}(f\hat{\wedge}_r g_1)J_1(g_2)\\
        &= \sum_{r=0}^{q-1}r! \binom{p}{r} \binom{q-1}{r}\left\{ J_{p+q-2r}((f\hat{\wedge}_r g_1)\wedge g_2) + (p+q-1-2r)J_{p+q-2r-2}((f\hat{\wedge}_r g_1)\wedge_1 g_2) \right\}\\
        &= \sum_{r=0}^{q-1}r! \binom{p}{r} \binom{q-1}{r}J_{p+q-2r}((f\hat{\wedge}_r g_1)\wedge g_2)\\
        &\quad + \sum_{r=1}^{q}(r-1)! \binom{p}{r-1} \binom{q-1}{r-1} (p+q-2r+1)J_{p+q-2r}((f\hat{\wedge}_{r-1} g_1)\wedge_1 g_2).
    \end{align*}
It remains to show that
$$q(f\hat{\wedge}_r g) = \frac{r(p+q-2r+1)}{p-r+1} (f \hat{\wedge}_{r-1} g_1)\wedge_1 g_2 + (q-r)(f\hat{\wedge}_r g_1)\wedge g_2. 
    $$
Note that
    $$
    g = \frac{1}{q}\sum_{i=1}^q (-1)^{q-i}g_1(t_1, \overset{i}{\Check{\dots}}, t_q)g_2(t_i).
    $$
    Thus,
    \begin{align}
        q(f\hat{\wedge}_r g) &= 
        \frac{1}{(p+q-2r)!}\sum_{\sigma \in S_{p+q-2r}}\sum_{i=1}^r \sgn(\sigma)(-1)^{q-i} \nonumber\\
        &\quad \times
        \int f(t_{\sigma(1)}, \dots, t_{\sigma(p-r)}, \overrightarrow{s_r})
        g_1(\overset{i}{\Check{\overleftarrow{s_r}}}, t_{\sigma(p-r+1)}, \dots, t_{\sigma(p+q-2r)})
        g_2(s_i)\ d\bf{s} \nonumber\\
        &+ \frac{1}{(p+q-2r)!}\sum_{\sigma \in S_{p+q-2r}}\sum_{i=r+1}^q \sgn(\sigma)(-1)^{q-i} \nonumber\\
        &\quad \times
        \int f(t_{\sigma(1)}, \dots, t_{\sigma(p-r)}, \overrightarrow{s_r})
        g_1(\overleftarrow{s_r}, t_{\sigma(p-r+1)}, \overset{i}{\Check{\dots}}, t_{\sigma(p+q-2r)})
        g_2(t_{\sigma(p-r+i-r)})\ d\bf{s} \nonumber\\
    &=
        \frac{(-1)^{q-r}r}{(p+q-2r)!}\sum_{\sigma \in S_{p+q-2r}}\sgn(\sigma) \nonumber\\
        &\quad \times
        \int f(t_{\sigma(1)}, \dots, t_{\sigma(p-r)}, \overrightarrow{s_r})
        g_1(\overleftarrow{s_{r-1}}, t_{\sigma(p-r+1)}, \dots, t_{\sigma(p+q-2r)})
        g_2(s_r)\ d\bf{s} \nonumber \nonumber\\
        &+\frac{q-r}{(p+q-2r)!}\sum_{\sigma \in S_{p+q-2r}}\sgn(\sigma) \nonumber\\
        &\quad \times
        \int f(t_{\sigma(1)}, \dots, t_{\sigma(p-r)}, \overrightarrow{s_r})
        g_1(\overleftarrow{s_r}, t_{\sigma(p-r+1)}, \dots, t_{\sigma(p+q-2r-1)})
        g_2(t_{\sigma(p+q-2r)})\ d\bf{s}, \label{4.1.1}
    \end{align}
    where we used the fact that the first part has anti-symmetry of $f$ and the second part has anti-symmetry about $t_1, \dots, t_{p+q-2r}$ for (\ref{4.1.1}).
    Moreover,
        \begin{align}
        (f\hat{\wedge}_{r-1} g_1)\wedge_1 g_2 &=
        \frac{1}{(p+q-2r+1)!}\sum_{\sigma \in S_{p+q-2r+1}}\sgn(\sigma) \nonumber\\
        &\quad\times \int f(t_{\sigma(1)}, \dots, t_{\sigma(p-r+1)}, \overrightarrow{s_{r-1}})
        g_1(\overleftarrow{s_{r-1}}, t_{\sigma(p-r+2)}, \dots, t_{\sigma(p+q-2r+1)})\nonumber\\
        &\qquad \times g_2(t_{p+q-2r+1})\ d{\bf s} dt_{p+q-2r+1}, \label{2.5}\\
        (f\hat{\wedge}_{r} g_1)\wedge g_2 &=
        \frac{1}{(p+q-2r+1)!}\sum_{\sigma \in S_{p+q-2r}}\sgn(\sigma)\nonumber\\
        &\quad\times \int f(t_{\sigma(1)}, \dots, t_{\sigma(p-r)}, \overrightarrow{s_{r}})
        g_1(\overleftarrow{s_{r}}, t_{\sigma(p-r+1)}, \dots, t_{\sigma(p+q-2r-1)})
        g_2(t_{\sigma(p+q-2r)})\ d{\bf s}.\nonumber
    \end{align}
    Note that $g_1\wedge_1 g_2=0$. If $\sigma^{-1}(p+q-2r+1)\notin \{ 1, \dots, p-r+1 \}$, the above integral will be zero. We define $\sigma'$ by
    $$\sigma' =
\left( \begin{array}{ccc}
1 & \cdots & p+q-2r \\
\sigma(1) & \overset{\sigma^{-1}(p+q-2r+1)}{\Check{\cdots}} & \sigma(p+q-2r+1) \\
\end{array}
\right)\in S_{p+q-2r}.$$
Let $\sgn(\sigma')$ be the sign of permuting $\sigma(1), \overset{\sigma^{-1}(p+q-2r+1)}{\Check{\dots}}, \sigma(p+q-2r+1)$ in ascending order. Then 
\begin{align*}
    &\sgn(\sigma)= (-1)^{p+q-2r+1-\sigma^{-1}(p+q-2r+1)}\sgn(\sigma'),\\
    &f(t_{\sigma(1)}, \dots, t_{\sigma(p-r+1)}, \overrightarrow{s_{r-1}})\\
    & = (-1)^{p-r+1-\sigma^{-1}(p+q-2r+1)}f(t_{\sigma'(1)}, \overset{\sigma^{-1}(p+q-2r+1)}{\Check{\ldots}}, t_{\sigma'(p-r+1)}, t_{p+q-2r+1}, \overrightarrow{s_{r-1}}).
\end{align*}
Therefore, multiplying both sides of (\ref{2.5}) by $r(p+q-2r+1)/(p-r+1)$, we have
\begin{align*}
&\frac{r(p+q-2r+1)}{p-r+1}(f\hat{\wedge}_{r-1}g_1)\wedge_1 g_2\nonumber \\
        &= \frac{r}{p-r+1}\cdot\frac{1}{(p+q-2r)!}\sum_{\sigma\in S_{p+q-2r+1}}\sgn(\sigma)\nonumber \\
        &\quad\times \int f(t_{\sigma(1)}, \dots, t_{\sigma(p-r+1)}, \overrightarrow{s_{r-1}})
        g_1(\overleftarrow{s_{r-1}}, t_{\sigma(p-r+2)}, \dots, t_{\sigma(p+q-2r+1)})\\
         &\qquad\times g_2(t_{p+q-2r+1})\ d{\bf s} dt_{p+q-2r+1}\\
    &= \frac{(-1)^{q+r}r}{(p+q-2r)!}\sum_{\sigma'\in S_{p+q-2r}}\sgn(\sigma')\\
        &\quad\times \int f(t_{\sigma'(1)}, \dots, t_{\sigma'(p-r+1)}, t_{p+q-2r+1}, \overrightarrow{s_{r-1}})
        g_1(\overleftarrow{s_{r-1}}, t_{\sigma'(p-r+2)}, \dots, t_{\sigma'(p+q-2r+1)})\\
        &\qquad\times g_2(t_{p+q-2r+1})\ d{\bf s} dt_{p+q-2r+1}.
\end{align*}
By renaming $t_{p+q-2r+1}$ to $s_r$, the desired result is obtained.
\end{proof}
We define the anti-symmetric Malliavin derivative and divergence on the Clifford algebra.
\begin{dfn}\label{D and delta}
    For every $n\in \mathbb{Z}_{\ge 0}$, let $v_n \in \Lambda_n $ and $ w_n\in L^2(\mathbb{R}_+^{n+1})$ such that $w_n(t,\cdot)\in \Lambda_n$ for $\mathrm{a.e.}\ t\ge 0$. We define the operators $\mathcal{D}_t : D^{-1}\Lambda_n \rightarrow L^2(\mathbb{R}_+)\otimes D^{-1}\Lambda_{n-1},\ \delta : L^2(\mathbb{R}_+)\otimes D^{-1}\Lambda_n \rightarrow D^{-1}\Lambda_{n+1}$ as follows:
    \begin{gather*}
        \mathcal{D}_t (J_n(v_n)) \coloneq nJ_{n-1}(v_n(t,\cdot)),\\
        \delta(J_n(w_n(t,\cdot))) \coloneq J_{n+1}(\hat{w}_n),
    \end{gather*}
    where $J_{-1}(\cdot)=0$. For every $N\in\mathbb{Z}_{\ge0}$, these definitions are extended linearly to the series $$\sum_{i=0}^N J_i(v_i),\ \sum_{i=0}^N J_i(w_i(t,\cdot)).$$
\end{dfn}
\begin{prop}\label{anti-Leibniz}
For every $p, q \in \mathbb{Z}_{\ge 0}$, let $f\in \Lambda_p,\ g\in \Lambda_q$. Then 
\begin{align*}
        \mathcal{D}_t(J_p(f)J_q(g))&= \mathcal{D}_t(J_p(f))J_q(g) + (-1)^pJ_p(f)\mathcal{D}_t(J_q(g))\\
        &= \mathcal{D}_t(J_p(f))J_q(g) + \beta(J_p(f))\mathcal{D}_t(J_q(g)).
    \end{align*}
\end{prop}
\begin{proof}
By using Proposition~\ref{anti-multiple formula} and Definition~\ref{D and delta}, we have
    \begin{align*}
        \mathcal{D}_t(J_p(f)J_q(g))&= \mathcal{D}_t\big( \sum_{r=0}^{p\wedge q} r!\binom{p}{r} \binom{q}{r}J_{p+q-2r}(f\hat{\wedge}_r g) \big)\\
        &= \sum_{r=0}^{p\wedge q} r!\binom{p}{r} \binom{q}{r}(p+q-2r)J_{p+q-2r-1} ( (f\hat{\wedge}_r g)(t,\cdot) ).
    \end{align*}
    In the same way,
        \begin{align*}
        \mathcal{D}_t(J_p(f)) J_q(g) &= pJ_{p-1}(f(t,\cdot)) J_q(g)\\
        &= \sum_{r=0}^{(p-1)\wedge q} r!\binom{p-1}{r} \binom{q}{r}pJ_{p+q-2r-1} (f(t,\cdot)\hat{\wedge}_r g)\\
        &= \sum_{r=0}^{p\wedge q} r!\binom{p}{r} \binom{q}{r}(p-r)J_{p+q-2r-1} (f(t,\cdot)\hat{\wedge}_r g),\\
        J_p(f)\mathcal{D}_t(J_q(q)) &= \sum_{r=0}^{p\wedge q} r!\binom{p}{r} \binom{q}{r}(q-r)J_{p+q-2r-1} (f\hat{\wedge}_r g(t,\cdot)).
    \end{align*}
    Here, we use techniques similar to those at the end of the proof of Proposition~\ref{anti-multiple formula}. For $\sigma\in S_{p+q-2r}$, we define
    $$\sigma' =
\Big( \begin{array}{ccc}
1 & \cdots & p+q-2r-1 \\
\sigma(0) & \overset{\sigma^{-1}(0)}{\Check{\cdots}} & \sigma(p+q-2r-1) \\
\end{array}
\Big)\in S_{p+q-2r-1}.$$
By using $\sigma'$, we can compute the integrand $f\hat{\wedge}_r g.$
    \begin{align*}
        &(p+q-2r)(f\hat{\wedge}_r g)(x_0, x_1, \dots, x_{p+q-2r-1})\\
        &=
        \frac{1}{(p+q-2r-1)!}\sum_{\sigma\in S_{p+q-2r}}\sgn(\sigma)\\
        &\quad\times \int f(x_{\sigma(0)}, x_{\sigma(1)}, \dots, x_{\sigma(p-r-1)},  \overrightarrow{s_r})g(\overleftarrow{s_r}, x_{\sigma(p-r)}, \dots, x_{\sigma(p+q-2r-1)} )\ d{\bf s}\\
        &= 
        \frac{1}{(p+q-2r-1)!}\big(\sum_{\substack{\sigma\in S_{p+q-2r} \\ \sigma^{-1}(0)\in \{ 0, \dots, p-r-1 \} }} + \sum_{\substack{\sigma\in S_{p+q-2r} \\ \sigma^{-1}(0)\in \{ p-r, \dots, p+q-2r-1 \} }}\big)\sgn(\sigma)\\
        &\quad\times \int f(x_{\sigma(0)}, x_{\sigma(1)}, \dots, x_{\sigma(p-r-1)},  \overrightarrow{s_r})g(\overleftarrow{s_r}, x_{\sigma(p-r)}, \dots, x_{\sigma(p+q-2r-1)} )\ d{\bf s}\\
        &=
        \frac{1}{(p+q-2r-1)!}\sum_{\substack{\sigma\in S_{p+q-2r} \\ \sigma^{-1}(0)\in \{ 0, \dots, p-r-1 \} }}(-1)^{\sigma^{-1}(0)}\sgn(\sigma')\\
        &\quad\times \int (-1)^{\sigma^{-1}(0)} f(x_0, x_{\sigma'(1)}, \overset{\sigma^{-1}(0)}{\Check{\dots}}, x_{\sigma'(p-r-1)},  \overrightarrow{s_r})g(\overleftarrow{s_r}, x_{\sigma'(p-r)}, \dots, x_{\sigma'(p+q-2r-1)} )\ d{\bf s}\\
        &+ 
        \frac{1}{(p+q-2r-1)!}\sum_{\substack{\sigma\in S_{p+q-2r} \\ \sigma^{-1}(0)\in \{ p-r, \dots, p+q-2r-1 \} }}(-1)^{\sigma^{-1}(0)}\sgn(\sigma')\\
        &\quad\times \int (-1)^{\sigma^{-1}(0)-(p-r)} f(x_{\sigma'(1)}, \dots, x_{\sigma'(p-r)},  \overrightarrow{s_r})g(\overleftarrow{s_r}, x_0, x_{\sigma'(p-r+1)}, \overset{\sigma^{-1}(0)}{\Check{\dots}}, x_{\sigma'(p+q-2r-1)} )\ d{\bf s}\\
        &=
        \frac{p-r}{(p+q-2r-1)!}\sum_{\sigma \in S_{p+q-2r-1}}\sgn(\sigma)\\
        &\quad\times \int f(x_0, x_{\sigma(1)}, \dots, x_{\sigma(p-r-1)},  \overrightarrow{s_r})g(\overleftarrow{s_r}, x_{\sigma(p-r)}, \dots, x_{\sigma(p+q-2r-1)} )\ d{\bf s}\\
        &+ 
        \frac{(-1)^p (q-r)}{(p+q-2r-1)!}\sum_{\sigma\in S_{p+q-2r-1}}\sgn(\sigma)\\
        &\quad\times \int f(x_{\sigma(1)}, \dots, x_{\sigma(p-r)},  \overrightarrow{s_r})g(x_0, \overleftarrow{s_r}, x_{\sigma(p-r+1)}, \dots, x_{\sigma'(p+q-2r-1)} )\ d{\bf s}.
    \end{align*}
    Similarly, we have
        \begin{align*}
        (p-r)f(t,\cdot)\hat{\wedge}_r g &= \frac{p-r}{(p+q-2r-1)!}\sum_{\sigma \in S_{p+q-2r-1}}\sgn(\sigma)\\
        &\quad\times \int f(t, x_{\sigma(1)}, \dots, x_{\sigma(p-r-1)},  \overrightarrow{s_r})g(\overleftarrow{s_r}, x_{\sigma(p-r)}, \dots, x_{\sigma(p+q-2r-1)} )\ d{\bf s},\\
        (q-r)f\hat{\wedge}_r g(t, \cdot) &= \frac{q-r}{(p+q-2r-1)!}\sum_{\sigma\in S_{p+q-2r-1}}\sgn(\sigma)\\
        &\quad\times \int f(x_{\sigma(1)}, \dots, x_{\sigma(p-r)},  \overrightarrow{s_r})g(t, \overleftarrow{s_r}, x_{\sigma(p-r+1)}, \dots, x_{\sigma'(p+q-2r-1)} )\ d{\bf s}.
    \end{align*}
    Therefore, the desired result is obtained.
    \end{proof}
\begin{prop}\label{anti integration by part}
For every $p, q \in \mathbb{Z}_{\ge 0}$, let $f \in L^2(\mathbb{R}_+^{p+1})$ such that $f(t, \cdot)\in \Lambda_p$ for $\mathrm{a.e.}\ t\ge 0$ and $ g\in \Lambda_q $. Then 
\begin{align*}
        m\big( \delta( J_p(f(t,\cdot)) ) J_{p+1}(g)\big) &= (-1)^p\int m\big(J_p(f(t, \cdot)) \mathcal{D}_t(J_{p+1}(g))\big)\ dt\\
        &= \int m\big(\beta(J_p(f(t, \cdot))) \mathcal{D}_t(J_{p+1}(g))\big)\ dt.
    \end{align*}
\end{prop}
\begin{proof}
Note that $m\big(J_p(h) J_p(k)\big) = p!\int h(\overrightarrow{x_p})k(\overleftarrow{x_p})\ d{\bf x}$ for $h,k\in\Lambda_p$. Therefore,
\begin{align*}
        m\big(\delta(J_p(f(t,\cdot))) J_{p+1}(g)\big)
        &=
        (p+1)!\int \frac{1}{p+1}\sum_{i=1}^{p+1}(-1)^{i-1}f(t_{p+1}, \overset{p-i+2}{\Check{\dots}}, t_1, t_i)g(t_1, \dots, t_{p+1})\ d{\bf t}\\
        &=
        (p+1)!\int \frac{1}{p+1}\sum_{i=1}^{p+1}(-1)^{i-1}f(t_{p+1}, \overset{p-i+2}{\Check{\dots}}, t_1, t_i)(-1)^{i-1}g(t_i, t_1, \overset{i}{\Check{\dots}}, t_{p+1})\ d{\bf t}\\
        &=
        (p+1)!\int  f(\overrightarrow{t_{p+1}})g(\overleftarrow{t_{p+1}})\ d{\bf t}\\
        &=
        (p+1)!\int  f(t, \overrightarrow{t_{p}})g(\overleftarrow{t_{p}}, t)\ d{\bf t} dt\\
        &=
        (p+1)!(-1)^p \int  f(t, \overrightarrow{t_{p}})g(t, \overleftarrow{t_{p}})\ d{\bf t} dt\\
        &= 
        (-1)^p\int m\big(J_p(f(t, \cdot)) \mathcal{D}_t(J_{p+1}(g))\big)\ dt.
    \end{align*}
    Consequently, we have the desired result.
\end{proof}
\begin{prop}
Let $p\in\mathbb{Z}_{\ge0}$, $h\in L^2(\mathbb{R}_+)$ and $f\in \Lambda_p$, and let $h\otimes J_p(f) \in L^2(\mathbb{R}_+)\otimes D^{-1}\Lambda_p$. Then
$$\{ \mathcal{D}_t, \delta \} h\otimes J_p(f) = h(t)J_p(f). $$
Formally, $\{ \mathcal{D}_t, \delta \} = \delta_t\otimes \mathrm{Id}_{L^2(\mathscr{C})}$, where $\delta_t$ is the Dirac delta function concentrated at $t$.
\end{prop}
\begin{proof}
The following formulas hold:
\begin{align*}
        J_1(h)J_p(f)&= J_{p+1}(h\wedge f)+pJ_{p-1}(h\wedge_1 f)\\
        &=
        \delta(h(t) J_p(f))+\int h(t)\mathcal{D}_t(J_p(f))\ dt,\\
        \mathcal{D}_t\circ \mathcal{D}_s (J_p(f))&= p(p-1)J_{p-2}(f(s, t, \cdot))\\
        &=  -p(p-1)J_{p-2}(f(t, s, \cdot))\\
        &= -\mathcal{D}_s\circ \mathcal{D}_t (J_p(f)).
    \end{align*}
    Using Proposition~\ref{anti-Leibniz}, we have
    \begin{align*}
        \mathcal{D}_t\circ \delta(h \otimes J_p(f)) &= h(t)J_p(f) - J_1(h)\mathcal{D}_t(J_p(f)) -\int h(s)\mathcal{D}_t\circ \mathcal{D}_s (J_p(f))\ ds,\\
        \delta \circ \mathcal{D}_t(h \otimes J_p(f))
        &=\delta(h \otimes \mathcal{D}_t(J_p(f)))\\
        &=
        J_1(h)\mathcal{D}_t(J_p(f)) -\int h(s)\mathcal{D}_s\circ \mathcal{D}_t (J_p(f))\ ds.
    \end{align*}
    Therefore, the desired result is obtained.
    \end{proof}
We investigate the domains of $\mathcal{D}, \delta$.
\begin{prop}
The operator $\mathcal{D}$ is closable as an operator from $L^2(\mathscr{C})$ to $L^2(\mathbb{R}_+)\otimes L^2(\mathscr{C})$.
\end{prop}
\begin{proof}
Let $\{ V_n \} \subset D^{-1}(\bigoplus_n \Lambda_n)$ be a sequence such that $ V_n \to 0$ in $L^2(\mathscr{C})$, $\mathcal{D}_{\cdot}(V_n) \to \eta_{\cdot}$ in $L^2(\mathbb{R}_+)\otimes L^2(\mathscr{C})$. For any $ h\in L^2(\mathbb{R}_+) $ and $W\in  D^{-1}(\bigoplus_n \Lambda_n)$, we have
 \begin{align*}
        \int m( \eta_t h(t)  W)\ dt
        &= \lim_{n\to \infty}\int m( \mathcal{D}_t (V_n) h(t) W)\ dt\\
        &= \lim_{n\to \infty}\int m(-\beta(V_n)\mathcal{D}_t(W)h(t) + V_n W\Psi(h))\ dt\\
        &= 0,
    \end{align*}
    where we used Proposition~\ref{anti integration by part} for the second line.  Since $L^2(\mathbb{R}_+)\otimes D^{-1}(\bigoplus_n \Lambda_n)$ is dense in $L^2(\mathbb{R}_+)\otimes L^2(\mathscr{C})$,  it follows that $\eta_{\cdot}=0.$
\end{proof}
The closure of $\mathcal{D}$ is also denoted by $\mathcal{D}$. $\mathrm{Dom}(\mathcal{D})$ is the closure of $D^{-1}(\bigoplus_n \Lambda_n)$ with respect to the norm 
$$\| F \|_{1,2} = [ m(| F |^2) + m(\| \mathcal{D}_{\cdot}(F) \|^2_{L^2(\mathbb{R}_+)} )]^{1/2}.$$
\begin{prop}
    Let $F\in L^2(\mathscr{C})$ be expressed as $F=\sum_{n=0}^\infty J_n(f_n)$. Then $F\in \mathrm{Dom}(\mathcal{D})$ if and only if 
    \begin{align*}
        m(\| \mathcal{D}_{\cdot}F \|_{L^2(\mathbb{R}_+)}^2)&= \sum_{n=1}^\infty n^2 \| J_{n-1}(f(t,\cdot)) \|^2_{L^2(\mathbb{R}_+) \otimes L^2(\mathscr{C})}\\
        &= \sum_{n=1}^\infty n \| f_n \|_{\Lambda_n}^2\\
        &= \sum_{n=1}^\infty n\cdot n! \int |f_n|^2\ d{\bf x} < \infty.
    \end{align*}
    For $F\in \mathrm{Dom}(\mathcal{D})$, we have $\mathcal{D}_t(F)=\sum_{n=1}^\infty nJ_{n-1}(f_n(t,\cdot))$.
\end{prop}
\begin{proof}
This is shown directly from the action of $\mathcal{D}_t$ on $D^{-1}(\bigoplus_n \Lambda_n)$.
\end{proof}
\begin{prop}
    The adjoint of $\mathcal{D}$ is an extension of $\delta$.
\end{prop}
\begin{proof}
    Let $F= \sum_{n=0}^\infty J_n(f_n)$ and let $g_n(t,\cdot)\in \Lambda_n$ be such that $g_n(t,\cdot)=0$ for all but finitely many $n$. We define $G_t= \sum_{n=0}^\infty J_n(g_n(t,\cdot))$. Then
        \begin{align*}
        \langle \mathcal{D}_\cdot F, G_{\cdot} \rangle_{L^2(\mathbb{R}_+)\otimes L^2(\mathscr{C})} &= \int m\big( (\mathcal{D}_t F)^* G_t \big)\ dt\\
        &= \int m\big( (\sum_n nJ_{n-1}(\overline{f_n(t, \overleftarrow{\cdot})}) )( \sum_n J_n(g_n(t,\cdot)) ) \big)\ dt\\
        &= \sum_{n=1}^\infty n!\int \overline{f_n(t, x_1, \dots, x_{n-1})}g_{n-1}(t, x_1, \dots, x_{n-1})\ d{\bf x}dt,\\
        \langle F, \delta (G_{t}) \rangle_{L^2(\mathscr{C})} 
        &= \sum_{n=1}^\infty n! \int\overline{f_n(t, x_1, \dots, x_{n-1})}\\
        &\quad\times \big( \sum_{i=1}^n \frac{(-1)^{i-1}}{n} g_{n-1}(x_1, \dots, x_{i-1}, t, x_{i}, \dots, x_{n-1}) \big)\ d{\bf x}dt\\
        &= \sum_{n=1}^\infty \sum_{i=1}^n (n-1)! \int \overline{f_n(x_1, \dots, x_{i-1}, t, x_{i}, \dots, x_{n-1})}\\
        &\quad\times g_{n-1}(x_1, \dots, x_{i-1}, t, x_{i}, \dots, x_{n-1})\ d{\bf x}dt\\
        &= \sum_{n=1}^\infty n!\int \overline{f_n(t, x_1, \dots, x_{n-1})}g_{n-1}(t, x_1, \dots, x_{n-1})\ d{\bf x}dt.
    \end{align*}
    Consequently, the desired result is obtained.
\end{proof}
By this proposition, it follows that $\delta$ is closed. The closure of $\delta$ is also denoted by $\delta$. $\mathrm{Dom}(\delta)$ is the set of $u \in L^2(\mathbb{R}_+)\otimes L^2(\mathscr{C})$ such that 
$$
| m( \langle \mathcal{D}_{\cdot}F, u \rangle_{L^2(\mathbb{R}_+)} ) | \le c\| F \|_{L^2(\mathscr{C})}
$$
for all $F \in \mathrm{Dom}(D)$, where $c$ is some constant depending on $u$. For $n \in \mathbb{N}$ and $f \in L^2(\mathbb{R}_+^n)$, the anti-symmetrization $\hat{f}$ of $f$ is defined by 
$$
\hat{f}(t_1, \dots, t_n)= \frac{1}{n!} \sum_{\sigma \in S_n}\sgn(\sigma) f(t_{\sigma(1)}, \dots, t_{\sigma(n)}).
$$
\begin{prop}
    Let $u \in L^2(\mathbb{R}_+)\otimes L^2(\mathscr{C})$ be such that $u = \sum_{n=0}^\infty J_n(f_n(t, \cdot))$. Then $u\in \mathrm{Dom}(\delta)$ if and only if 
    $$
    \| \sum_{n=0}^{\infty} J_{n+1}(\hat{f}_n) \|^2_{L^2(\mathscr{C})} = \sum_{n=0}^\infty (n+1)! \int |\hat{f}_n|\ d{\bf x} < \infty.
    $$
    For $u\in \mathrm{Dom}(\delta)$, we have $\delta(u)=\sum_{n=0}^\infty J_{n+1}(\hat{f}_n)$.
\end{prop}
\begin{proof}
    The necessity is immediate from the definition. We show the sufficiency. Let $\sum_{n=0}^\infty J_{n+1}(\hat{f}_n) = V$. Then
    $$
    \langle \mathcal{D}_\cdot\sum_{n=0}^N J_n(g_n), u \rangle_{L^2(\mathbb{R}_+)\otimes L^2(\mathscr{C})} = \langle \sum_{n=0}^N J_n(g_n), \delta(u) \rangle_{L^2(\mathbb{R}_+)\otimes L^2(\mathscr{C})}
    $$
    for all $N\ge 0$. Therefore, by the density of $D^{-1}(\bigoplus_n \Lambda_n)$, 
    $$
    | \langle \mathcal{D}_\cdot F, u \rangle_{L^2(\mathbb{R}_+)\otimes L^2(\mathscr{C})} | \le \| F \|_{L^2(\mathscr{C})} \| V \|_{L^2(\mathscr{C})}
    $$
    for any $F\in \mathrm{Dom}(\mathcal{D})$. This implies that $u\in \mathrm{Dom}(\delta)$ and $V=\delta(u)$.
\end{proof}
Next, we discuss the relationship between $\mathcal{D},\ \delta$, and conditional expectations. Let $A$ be a Borel set in $\mathbb{R}_+$, and let $\mathscr{C}_A$ denote the $W^*$-subalgebra generated by $\{ \Psi(u) \mid u \in L^2(\mathbb{R}_+), \supp(u)\subset A \}$.
\begin{lem}\label{projection}
Let $F\in L^2(\mathscr{C})$ be expressed as $F= \sum_{n=0}^\infty J_n(f_n)$. Then 
$$
    m(F|\mathscr{C}_A) = \sum_{n=0}^\infty J_n(f_n\mathbbm{1}_A^{\otimes n}).
    $$
\end{lem}
\begin{proof}
    The conditional expectation $m(\cdot|\mathscr{C}_A)$ and $D^{-1}\Gamma(\mathbbm{1}_A)D$ act as a projection from $L^2(\mathscr{C})$ to $L^2(\mathscr{C}_A)$.
\end{proof}
\begin{prop}\label{conditonal}
    Let $F\in \mathrm{Dom}(\mathcal{D})$ and let $A$ be a Borel set. Then $m(F|\mathscr{C}_A)\in \mathrm{Dom}(\mathcal{D})$, and 
    $$
    \mathcal{D}_t(m(F|\mathscr{C}_A))= m(\mathcal{D}_t F| \mathscr{C}_A)\mathbbm{1}_A(t).
    $$
\end{prop}
\begin{proof}
    By using Lemma~\ref{projection},
    $$
        \mathcal{D}_t(m(F|\mathscr{C}_A))= \sum_{n=1}^\infty nJ_{n-1}(f_n(t, \cdot)\mathbbm{1}_A^{\otimes n-1})\mathbbm{1}_A(t) = m(\mathcal{D}_t F|\mathscr{C}_A)\mathbbm{1}_A(t).
    $$
\end{proof}
\begin{lem}\label{delta}
    Let $A\in\mathcal{B}_0$, and let $F\in L^2(\mathscr{C}_{A^c})$. Then $\mathbbm{1}_A F \in \mathrm{Dom}(\delta)$, and 
    \begin{align}
    \delta(\mathbbm{1}_A F) = \Psi(\mathbbm{1}_A)F.\label{4.12}
   \end{align}
\end{lem}
\begin{proof}
Let $F= \sum_{n=0}^\infty J_n(f_n)$. First we prove $\mathbbm{1}_A F\in \mathrm{Dom}(\mathcal{D})$. Observe that
\begin{align*}
        \sum_{n=0}^\infty (n+1)! \int |\widehat{\mathbbm{1}_A f_n}|^2\ d{\bf x}
        &= \sum_{n=0}^\infty (n+1)!\int \sum_{\sigma, \tau\in S_{n+1}} \frac{1}{((n+1)!)^2} \sgn(\sigma) \sgn(\tau) \mathbbm{1}_A(x_{\sigma(1)})\\
        &\quad\times \overline{f_n( x_{\sigma(2)}, \dots, x_{\sigma(n+1)} )}
        \mathbbm{1}_A(x_{\tau(1)})f_n( x_{\tau(2)}, \dots, x_{\tau(n+1)} )\ d{\bf x}\\
        &= \sum_{n=0}^\infty n! \int \sum_{\substack{\sigma, \tau \in S_{n+1} \\ \sigma(1)=\tau(1) }} \sgn(\sigma)\sgn(\tau) \mu(A) \frac{1}{n+1}\frac{1}{(n!)^2}\\
        &\quad\times \overline{f_n( x_{\sigma(2)}, \dots, x_{\sigma(n+1)} )}f_n( x_{\tau(2)}, \dots, x_{\tau(n+1)} )\ d{\bf x},
\end{align*}
where we used the fact that $\mathrm{supp}(f_n)\subset \Pi^n A^c$. The last expression is seen to be equal to
\begin{align*}
        &\sum_{n=0}^\infty n! \int \sum_{\sigma, \tau \in S_n } \sgn(\sigma)\sgn(\tau) \mu(A) \frac{1}{(n!)^2}\overline{f_n( x_{\sigma(1)}, \dots, x_{\sigma(n)} )}f_n( x_{\tau(1)}, \dots, x_{\tau(n)} )\ d{\bf x}\\
        &= \mu(A)\sum_{n=0}^\infty n! \int |f_n|^2\ d{\bf x}\\
        &= \mu(A)\| F \|_{L^2(\mathscr{C})}.
    \end{align*}
The above observations imply $\mathbbm{1}_A F\in \mathrm{Dom}(\mathcal{D})$.
To show formula~(\ref{4.12}), let $F\in \mathrm{Dom}(\mathcal{D})$.
Note that $\delta(\mathbbm{1}_A F) = \Psi(\mathbbm{1}_A)F - \int \mathbbm{1}_A(s)\mathcal{D}_s(F)\ ds$. By using Proposition~\ref{conditonal}, we have $ \mathcal{D}_tF = \mathcal{D}_t(m(F|\mathscr{C}_{A^c}))= m(\mathcal{D}_t F| \mathscr{C}_{A^c})\mathbbm{1}_{A^c}(t)$, and hence (\ref{4.12}) follows. The general case follows from the density of $\mathrm{Dom}(\mathcal{D})$ and the fact that $\delta$ is closed.
\end{proof}
\begin{prop}\label{coinside with Ito-Clifford}
    $\mathfrak{H}[0,1] \subset \mathrm{Dom}(\delta)$, and the operator $\delta$ restricted to $\mathfrak{H}[0,1]$ coincides with the It\^{o}-Clifford integral:
     $$
    \delta(u) = \int_0^1 d\Psi_t u_t
    $$
    for every $u\in\mathfrak{H}[0,1]$.
\end{prop}
\begin{proof}
    Note that the set of simple adapted processes is dense in $\mathfrak{H}[0,1]$ {\cite[Thm~3.9]{BSW1}}. Let $u$ be a simple adapted process such that $u_t = \sum_{j=1}^n F_j \mathbbm{1}_{[t_j, t_{j+1} )}(t)$, where $F_j\in L^2(\mathscr{C}_{t_j})$, $0\le t_1 < \cdots < t_{n+1} \le 1$. By Lemma~\ref{delta}, $u\in \mathrm{Dom}(\delta)$ and
     $$
    \delta(u) = \sum_{j=1}^n (\Psi_{t_{j+1}}-\Psi_{t_j})F_j=\int_0^1 d\Psi_t u_t.
    $$
    For $u\in \mathfrak{H}[0,1]$, consider a sequence of simple adapted processes $u^n$ converging to $u$. Since $\int_0^1 d\Psi_t u^n_t$ converges to $\int_0^1 d\Psi_t u_t$ in $\mathfrak{H}[0,1]$ {\cite[Thm~3.10, Def.~3.11]{BSW1}}, and $\delta$ is closed, it follows that $\delta(u)=\int_0^1 d\Psi_t u_t$.
\end{proof}
\begin{thm}[The anti-symmetric Clark-Ocone formula]\label{The anti-symmetric Clark-Ocone formula}
    Let $F\in \mathrm{Dom}(\mathcal{D})$. Then 
    $$
    F=m(F)+\int_0^1 d\Psi_t m(\mathcal{D}_tF|\mathscr{C}_t).
    $$
\end{thm}
\begin{proof}
    For every $p \in \mathbb{N}$, let $f \in \Lambda_p$ be such that $\mathrm{supp}(f)\subset [0,1]^p$. Since  $\sum_{i=1}^p \mathbbm{1}_{[0,x_i]^{p-1}}(x_1, \overset{i}{\Check{\dots}}, x_p) = 1$ for $\mathrm{a.e.}\ {\bf x}$ for $p\ge 2$, by defining $\mathbbm{1}_{[0,x_1]^{0}} \coloneq 1$, we have
    \begin{align*}
        \delta\big( m(\mathcal{D}_t(J_p(f))|\mathscr{C}_t) \big)
        &= p\int \frac{1}{p}\sum_{i=1}^p (-1)^{i-1}f(x_i, x_1, \overset{i}{\Check{\dots}}, x_p)\mathbbm{1}_{[0,x_i]^{p-1}}(x_1, \overset{i}{\Check{\dots}}, x_p):\psi:d{\bf x}\\
        &= \int f(x_1, \dots, x_p)\sum_{i=1}^p \mathbbm{1}_{[0,x_i]^{p-1}}(x_1, \overset{i}{\Check{\dots}}, x_p) :\psi:d{\bf x}\\
&=J_p(f).
    \end{align*}
    For $p=0$, we have $\delta\big( m(\mathcal{D}_t(J_0(f))|\mathscr{C}_t) \big)=0$.
    Thus, for $F=\sum_{p=0}^\infty J_p(f_p)\in \mathrm{Dom}(\mathcal{D})$, by applying Proposition~\ref{coinside with Ito-Clifford}, we obtain
    $$
    \int_0^1 d\Psi_t m(\mathcal{D}_tF|\mathscr{C}_t) =\sum_{p=0}^\infty \delta\big( m(\mathcal{D}_t(J_p(f_p))|\mathscr{C}_t) \big) = \sum_{p=1}^\infty J_p(f_p) = F-m(F).
    $$
\end{proof}

% 4.2
\subsection{Applications}\label{subsection_application}
In this subsection, we deal with applications of Malliavin calculus on the Clifford algebra, specifically focusing on a concentration inequality, logarithmic Sobolev inequality, and the fourth-moment theorem. For concentration inequalities, we obtain results similar to those in \cite{HP02}, which do not demand the use of the Leibniz rule. For the logarithmic Sobolev inequality, only a limited result is obtained. Regarding the fourth-moment theorem, we show that there is the property differing from that of the Brownian motion on the Clifford algebra. Throughout this subsection, We use $ \| \cdot \|_p $ for $ \| \cdot \|_{L^2(\mathbb{R}_+^p)} $. We will begin by presenting the results related to concentration inequalities.
\begin{thm}[cf. {\cite[Lem.~3.1, Prop.~3.3]{HP02}}\label{HP02 ineq.} The concentration inequality]\label{concentration}
    Let  $F\in \mathrm{Dom}(\mathcal{D})\cap L^\infty (\mathscr{C})$ be a self-adjoint operator such that $\int \| \mathcal{D}_tF \|^2_\infty\ dt < \infty$. The spectral decomposition is written as $F=\int_{-\infty}^\infty \lambda\ dE_\lambda$. If $$h(s)=\int\| \mathcal{D}_tF \|_{L^\infty(\mathscr{C})} \| e^{-sF} \mathcal{D}_t e^{sF}\|_{L^\infty(\mathscr{C})}\ dt$$ is a monotonic function on $[0,\infty)$, then
    \begin{equation}\label{m ineq}
    m\left( E([ m(F)+x, \infty )) \right)\le \exp\left(-\int_0^x h^{-1}(s)\ ds\right)
    \end{equation}
    for every $x\ge 0$.
\end{thm}
\begin{proof}
Thanks to Theorem~\ref{The anti-symmetric Clark-Ocone formula}, the proof proceeds in the same way as in \cite{HP02}.
\end{proof}
The right-hand side of (\ref{m ineq}) can be evaluated by a more explicit form.
\begin{cor}
    Under the same assumptions of Theorem~\ref{concentration}, let $$A= \frac{\int\| \mathcal{D}_tF\|^2_{L^\infty (\mathscr{C})}\ dt}{2\|F\|_{L^\infty (\mathscr{C})}}.$$ Then
    $$
    m\left( E([ m(F)+x, \infty )) \right)\le \exp\left( -\frac{x}{2\| F \|_{L^\infty (\mathscr{C})}}\left(W \left( \frac{x}{A} \right)-1+\frac{1}{W\left( \frac{x}{A} \right)}\right) \right),
    $$
    where $W$ is the Lambert W function.
\end{cor}
\begin{proof}
    Note that $\| \beta(F) \|_{L^\infty (\mathscr{C})}= \| F \|_{L^\infty (\mathscr{C})}$. The following inequalities hold:
    \begin{gather*}
        \| \mathcal{D}_t e^{sF} \|_{L^\infty (\mathscr{C})}\le s\| \mathcal{D}_tF \|_{L^\infty (\mathscr{C})} e^{s\| F \|_{L^\infty (\mathscr{C})}},\\
        \| e^{-sF} \|_{L^\infty (\mathscr{C})} \le e^{s\| F \|_{L^\infty (\mathscr{C})}}.
    \end{gather*}
    Therefore, 
    $$
    h(s)\le \int \| \mathcal{D}_tF \|^2_{L^\infty (\mathscr{C})}\ dt \cdot se^{2s\| F \|_{L^\infty (\mathscr{C})}}\eqqcolon k(s).
    $$
    Given the above, we have
    $$
    -h^{-1}(s)\le -k^{-1}(s)=-\frac{1}{2\| F \|_{L^\infty (\mathscr{C})}}W\left( \frac{s}{A} \right).
    $$
\end{proof}
Next, we present the results concerning the logarithmic Sobolev inequality.
\begin{lem}
    Let $f$ be a continuous function on $\mathbb{R}$ and $0\neq z\in L^2(\mathbb{R}_+)$. Then
    $$
    f(\Psi(z))=\frac{f(\| z \|_1) + f(-\| z \|_1)}{2} + \frac{\Psi(z)\left( f(\| z \|_1) - f(-\| z \|_1) \right)}{2\| z \|_1}.
    $$
\end{lem}
\begin{proof}
    The case where $f$ is a monomial can be easily shown. For polynomial $f$, the result follows from the linearity of the formula with respect to $f$. By using the Weierstrass approximation theorem on $[-\|z\|_1, \|z\|_1]$, the result holds for a general $f$.
\end{proof}
In particular, 
$$
f(\Psi_t)=\frac{f(t^{1/2}) + f(-t^{1/2})}{2} + \frac{\Psi_t\left( f(t^{1/2}) - f(-t^{1/2}) \right)}{2t^{1/2}}.
$$
Let $f$ be differentiable. Taking the expectation and differentiating, we get
$$
\frac{dm\left( f(\Psi_t) \right)}{dt} =\frac{1}{4t^{1/2}}\left( f'(t^{1/2}) - f'(-t^{1/2}) \right).
$$
For 
$$
M_t= m(f^2(\Psi_1)| \mathscr{C}_t)= a\Psi_t + b,
$$
where $a= (f^2(1)-f^2(-1))/2, b= (f^2(1)+f^2(-1))/2$, we have
$$
\frac{dm\left( M_t\log M_t \right)}{dt}=\frac{a}{4t^{1/2}}\log\frac{b+at^{1/2}}{b-at^{1/2}}.
$$
Thus,
\begin{equation}\label{log1}
m\left( M_1\log M_1 \right)-M_0\log M_0 = \int_0^1 \frac{a}{4t^{1/2}}\log\frac{b+at^{1/2}}{b-at^{1/2}}\ dt.
\end{equation}
By setting $f=1+sx$, where $0\le s \le1$, in (\ref{log1}), we can reproduce a calculation process in the proof of {\cite[Thm~$3$]{Gro75b}}.
\begin{prop}[cf. {\cite[Thm~3]{Gro75b}}]
Let $f : \{ -1, 1 \} \to \mathbb{C}$. It holds that
    $$
    m(|f(\Psi_1)|^2 \log |f(\Psi_1)|^2) - m(|f(\Psi_1)|^2) \log m(|f(\Psi_1)|^2) \le 2\log 4\cdot\|\mathcal{D}_{\cdot} f(\Psi_1) \|^2_{L^2(\mathbb{R}_+)\otimes L^2(\mathscr{C})}.
    $$
\end{prop}
\begin{rem}
    This inequality does not provide a good estimate. By {\cite[Thm~3]{Gro75b}}, we can show that the right-hand side of this inequality becomes $ 2\|\mathcal{D}_{\cdot} f(\Psi_1) \|^2_{L^2(\mathbb{R}_+)\otimes L^2(\mathscr{C})}$. 
\end{rem}
\begin{proof}
    Using Theorem~\ref{The anti-symmetric Clark-Ocone formula}, we calculate $a^2$.
    \begin{align*}
        a^2 &= m\left( \left| \int_0^1 d\Psi_s m(\mathcal{D}_sf^2(\Psi_1)|\mathscr{C}_s) \right|^2 \right)\\
        &= m\left( \left| \int_0^1 d\Psi_s m(\mathcal{D}_sf(\Psi_1)\cdot f(\Psi_1) + \beta(f(\Psi_1))\cdot \mathcal{D}_sf(\Psi_1) |\mathscr{C}_s) \right|^2 \right)\\
        &=  \int_0^1 m\left( \left| m(\mathcal{D}_sf(\Psi_1)\cdot f(\Psi_1) + \beta(f(\Psi_1))\cdot \mathcal{D}_sf(\Psi_1) |\mathscr{C}_s) \right|^2 \right)\ ds,
    \end{align*}
    where we used the isometry for the It\^o-Clifford integral. By applying the conditional Schwarz inequality and using 
    the unitarity of $\beta$, the triangle inequality of $\|\cdot\|_{L^2(\mathscr{C})}$, and the Schwarz inequality, the integrand with respect to $s$ is dominated by
    $$
        4\|f(\Psi_1)\|_{L^2(\mathscr{C})}^2\cdot \|\mathcal{D}_{s}f(\Psi_1)\|_{L^2(\mathscr{C})}^2
        = 4b \|\mathcal{D}_{s}f(\Psi_1)\|_{L^2(\mathscr{C})}^2.
    $$
    Thus, it suffices to verify the following inequality for the right-hand side of (\ref{log1}):
    \begin{equation}\label{log2}
        \int_0^1 \frac{a}{4t^{1/2}}\log\frac{b+at^{1/2}}{b-at^{1/2}}\ dt\le \frac{a^2}{b}\log2.
    \end{equation}
    Note that
    \begin{align*}
        \int_0^1 \frac{a}{4t^{1/2}}\log\frac{b+at^{1/2}}{b-at^{1/2}}\ dt &= \int_0^1 \frac{a}{2}\log\frac{b+ax}{b-ax}\ dx\\
        &= \int_0^1 \int_0^x \frac{a^2b}{b^2-a^2y^2}\ dydx.
    \end{align*}
    (\ref{log2}) follows from the next inequality:
    \begin{align*}
        \int_0^1 \int_0^x \frac{a^2b}{b^2-a^2y^2}\ dydx &= \frac{a^2}{b}\int_0^1 \int_0^x \frac{1}{1-\left( \frac{a}{b}\right)^2 y^2}\ dydx\\
        &\le \frac{a^2}{b}\int_0^1 \int_0^x \frac{1}{1-y^2}\ dydx\\
        &= \frac{a^2}{b}\log2.
    \end{align*}
\end{proof}
Finally, we discuss the fourth-moment theorem. To begin with, we will mention the background of this theorem, referring to \cite{NP12}.
\begin{lem}[{\cite[Lem.~3.1.2]{NP12}} Stein's lemma]
    A real-valued random variable $N$ has the standard Gaussian distribution $\gamma$ (this is also written by $N\sim\mathcal{N}(0,1)$) if and only if, for every differentiable function $f : \mathbb{R} \to \mathbb{R}$ such that $f' \in L^1(\gamma)$, $N f (N)$ and $f' (N)$ are integrable and $$
    \mathbb{E}(N f (N)) = \mathbb{E}( f' (N)).
    $$
\end{lem}
Using Stein's equation {\cite[Def.~3.2.1]{NP12}}, we can show that some distances (e.g. the total variation distance) between the normal distribution and the random variable $F$ are bounded by
$$
\sup_{f\in \mathcal{F}}| \mathbb{E}(f'(F)) -\mathbb{E}(Ff(F)) |,
$$
where $\mathcal{F}$ is an appropriate family of functions.
Moreover,
$$
\sup_{f\in \mathcal{F}}\left| \mathbb{E}(f'(F)) -\mathbb{E}(Ff(F)) \right|\le \sup_{{f\in \mathcal{F},x\in \mathbb{R}}}|f'(x)|\cdot \mathbb{E}(|1- \langle \mathcal{D}_\cdot F, \mathcal{D}_\cdot R^{-1}F \rangle_{L^2(\mathbb{R}_+)}|),
$$
where $R=\delta\circ\mathcal{D}$. This suggests that the term $ 1- \langle \mathcal{D}_\cdot F, \mathcal{D}_\cdot R^{-1}F \rangle_{L^2(\mathbb{R}_+)} $ could be used to investigate the properties of the distribution of $F$.
The following statement is a result in usual Malliavin calculus.
\begin{thm}[{\cite[Thm~5.2.7]{NP12}} The fourth-moment theorem]\label{normal FM thm}
Let $q\ge2$ and let $F_n = I_q (f_n)$, $n\in\mathbb{N}$, be a sequence of random variables belonging to the $q$-th chaos (so that $f_n \in L^2_s(\mathbb{R}^q_+)$). Moreover, assume that $\mathbb{E}(F^2_n ) \to \sigma^2 > 0$ as $n \to \infty$. Then, as $n \to \infty$, the following six assertions are equivalent:
    \begin{enumerate}
        \item $F_n$ converges in distribution to $N \sim \mathcal{N}(0, \sigma^2)$;
        \item $\mathbb{E}(F^4_n ) \to 3\sigma^4 = \mathbb{E}(N^4)$;
        \item $\langle \mathcal{D}_\cdot F_n, \mathcal{D}_\cdot R^{-1}F_n \rangle_{L^2(\mathbb{R}_+)} \to 1$ in $L^2(\Omega)$;
        \item $\mathrm{Var}\left( \| \mathcal{D}_\cdot F_n \|_{L^2(\mathbb{R}_+)}^2\right) \to 0$;
        \item $\|f_n \tilde{\otimes}_r f_n\|_{2q-2r} \to 0$ for all $r = 1,\dots, q-1$;
        \item $\|f_n \otimes_r f_n\|_{2q-2r} \to 0$ for all $r = 1,\dots, q-1$.
    \end{enumerate}
\end{thm}
Although there is not assertion $(\mathrm{iii})$ in \cite[Thm~5.2.7]{NP12}, we can prove that assertion $(\mathrm{iii})$ is a necessary and sufficient condition for this theorem in a manner similar to the proof of Theorem~\ref{claim1}.
Theorem~\ref{normal FM thm} means that whether random variables belonging to the chaos of order $q\ge2$ converge to a normal distribution is determined by the fourth moment. Below, we will show that such a phenomenon does not occur in the case of fermions.
\begin{lem}\label{wedge lemma}
    Let every $p,q \in \mathbb{Z}_{\ge0}$, let $f\in \Lambda_p,\ g\in \Lambda_q$. The following hold:
    \begin{gather*}
        f\hat{\wedge}_r g = (-1)^{pq+r}g\hat{\wedge}_r f,\\
        \int \overline{f(t,\overleftarrow{\cdot})}\wedge_r f(t,\cdot)\ dt =\overline{\overleftarrow{f}}\wedge_{r+1} f.
    \end{gather*}
\end{lem}
\begin{proof}
We can prove this lemma by the following calculations:
    \begin{align*}
        f\hat{\wedge}_r g 
        &= \frac{1}{(p+q-2r)!}\sum_{\sigma\in S_{p+q-2r}}\int\sgn(\sigma) f(x_{\sigma(1)}, \dots, x_{\sigma(p-r)}, \overrightarrow{s_r})g(\overleftarrow{s_r}, x_{\sigma(p-r+1)}, \dots, x_{\sigma(p+q-2r)})\ d{\bf s}\\
        &= \frac{1}{(p+q-2r)!}\sum_{\sigma\in S_{p+q-2r}}\int\sgn(\sigma) (-1)^{p(p-1)/2}(-1)^{q(q-1)/2} g(x_{\sigma(p+q-2r)}, \dots, x_{\sigma(p-r+1)}, \overrightarrow{s_r})\\
        &\quad\times f(\overleftarrow{s_r}, x_{\sigma(p-r)}, \dots, x_{\sigma(1)})\ d{\bf s}\\
        &= \frac{1}{(p+q-2r)!}\sum_{\sigma\in S_{p+q-2r}}\int\sgn(\sigma) (-1)^{p(p-1)/2}(-1)^{q(q-1)/2}(-1)^{(p+q-2r)(p+q-2r-1)/2} \\
        &\quad\times g(x_{\sigma(1)}, \dots, x_{\sigma(p-r)}, \overrightarrow{s_r})f(\overleftarrow{s_r}, x_{\sigma(p-r+1)}, \dots, x_{\sigma(p+q-2r)})\ d{\bf s}\\
        &= (-1)^{pq+r}g\hat{\wedge}_{r} f,
    \end{align*}
    where we used the fact that $\overleftarrow{g\hat{\wedge}_r f}=(-1)^{(p+q-2r)(p+q-2r-1)/2}g\hat{\wedge}_r f$ for the third equality; similarly, we have
    \begin{align*}
        \int \overline{f(t,\overleftarrow{\cdot})}\wedge_r f(t,\cdot)(x_1, \dots, x_{2q-2-2r})\ dt
        &= \int \overline{f(t,\overleftarrow{s_r}, x_{q-1-r}, \dots, x_1)}f(t,\overleftarrow{s_r}, x_{q-r}, \dots, x_{2q-2-2r})\ d{\bf s}dt\\
        &= \int \overline{\overleftarrow{f}(x_1,\dots,x_{q-1-r},\overrightarrow{s_r}, t)}f(t,\overleftarrow{s_r}, x_{q-r}, \dots, x_{2q-2-2r})\ d{\bf s}dt\\
        &= \overline{\overleftarrow{f}}\wedge_{r+1} f(x_1, \dots, x_{2q-2-2r}).
    \end{align*}
\end{proof}
\begin{thm}\label{claim1}
    Let $q\ge2$ and let $F_n = J_q (f_n)$, $n\in\mathbb{N}$, be a sequence of random variables belonging to the $q$-th chaos (so that $f_n \in \Lambda_q$). Moreover, assume that $m(F_n^2) \to 1$ as $n \to \infty$. Then, as $n \to \infty$, the following three assertions are equivalent:
    \begin{enumerate}
        \item $\langle \mathcal{D}_\cdot F_n, \mathcal{D}_\cdot R^{-1}F_n \rangle \to 1$ in $L^2(\mathscr{C})$;
        \item $\|f_n \hat{\wedge}_r f_n\|_{2q-2r} \to 0$, for all $r = 1,\dots, q - 1$ such that $q+r$ is even;
        \item $\mathrm{Var}\left( \| \mathcal{D}_\cdot F_n \|^2 \right) \to 0$.
    \end{enumerate}
    Regarding the above conditions, $\langle f_{\cdot}, g_{\cdot} \rangle = \int f_t^* g_t\ dt$, where $f_t, g_t \in L^2(\mathscr{C})$ for a.e. $t\ge 0$, $\|\cdot\|^2 = \langle \cdot, \cdot\rangle$, and $\mathrm{Var}(F)=m(F^* F)-|m(F)|^2$ for $F\in L^2(\mathscr{C})$.
\end{thm}
\begin{proof}
Using Lemma~\ref{wedge lemma} and the fact that $F^* = F$ implies $\overline{\overleftarrow{f}} = f$, we can show the following equality:
\begin{align*}
        \|\mathcal{D}_\cdot F\|^2 &= \int(\mathcal{D}_t F)^* \mathcal{D}_t F\ dt\\
        &= q^2\sum_{r=0}^{q-1}r!\binom{q-1}{r}^2 J_{2q-2-2r}\left( \int \overline{f(t,\overleftarrow{\cdot})}\wedge_r f(t,\cdot)\ dt \right)\\
        &= q^2\sum_{r=0}^{q-1}r!\binom{q-1}{r}^2 J_{2q-2-2r}\left(\overline{\overleftarrow{f}}\hat{\wedge}_{r+1} f \right)\\
        &= q^2\sum_{r=1}^{q}(r-1)!\binom{q-1}{r-1}^2 J_{2q-2r}\left(f\hat{\wedge}_{r} f \right)\\
        &= qq!\|f\|_q^2+q^2\sum_{r=1}^{q-1}(r-1)!\binom{q-1}{r-1}^2 J_{2q-2r}\left(f\hat{\wedge}_{r} f \right)\\
        &= qm(F^2)+q^2\sum_{r=1}^{q-1}(r-1)!\binom{q-1}{r-1}^2 J_{2q-2r}\left(f\hat{\wedge}_{r} f \right)\\
        &= qm(F^2)+q^2\sum_{\substack{r=1 \\ q+r:\text{even}}}^{q-1}(r-1)!\binom{q-1}{r-1}^2 J_{2q-2r}\left(f\hat{\wedge}_{r} f \right).
\end{align*}
Thus,
$$
    \mathrm{Var}\left( \| \mathcal{D}_\cdot F \|^2 \right) = \sum_{\substack{r=1 \\ q+r:\text{even}}}^{q-1} r^2(r!)^2\binom{q}{r}^4(2q-2r)!\| f\hat{\wedge}_r f \|_{2q-2r}^2.
$$
Using the following equality:
$$
        \langle \mathcal{D}_\cdot F, \mathcal{D}_\cdot R^{-1}F \rangle = \frac{1}{q}\langle \mathcal{D}_\cdot F, \mathcal{D}_\cdot F \rangle = \frac{1}{q}\|\mathcal{D}_\cdot F\|^2,
$$
we obtain the desired results.
\end{proof}
\begin{rem}
    The reason why assertion $(\mathrm{i})$ in Theorem~\ref{claim1} can be considered as an expression of distributional convergence is that the following inequality holds (cf. \cite[Lem.~3.2]{BSW2}):
    $$
    \| e^{\sqrt{-1} tF}-e^{\sqrt{-1} t\Psi(z)} \|_{L^2(\mathscr{C})} \le |t| \| F -\Psi(z)\|_{L^2(\mathscr{C})} = |t|\left|1 - m\big(\langle \mathcal{D}_\cdot F, \mathcal{D}_\cdot R^{-1}F \rangle \big)\right|
    $$
    for $F=F^*\in \mathrm{Dom}(\mathcal{D})$ and $z\in L^2(\mathbb{R}_+)$ such that $\|z\|_1 = 1$.
    We are not sure how appropriate it is to treat the convergence of the characteristic functions of $F_n$ in $L^2(\mathscr{C})$ as distributional convergence. However, the above inequality provides a warrant for considering that assertion $(\mathrm{i})$ in Theorem~\ref{claim1}.
\end{rem}
\begin{lem}\label{anti-symmetric norm}
    For every $q \ge 2$ and $r \in \{ 0,\dots, q-1 \} $, let $f\in \Lambda_q$ satisfy $\overleftarrow{f}=f,\ \overline{f}=f$. Then
    \begin{equation}
        \| f\hat{\wedge}_r f \|^2_{2q-2r}= \frac{1+(-1)^{q+r}}{2(q-r)^2}\int \| f(t,\cdot)\hat{\wedge}_r f \|^2_{2q-2r-1}\ dt.\label{anti-symmetric norm eq.}
    \end{equation}
\end{lem}
\begin{proof}
We compute
    \begin{align*}
        \| f\hat{\wedge}_r f \|^2_{2q-2r} 
        &= \frac{1}{((2q-2r)!)^2}\sum_{\sigma, \rho \in S_{2q-2r}} \sgn(\sigma)\sgn(\rho)\\
        &\quad\times \int f(x_{\sigma(1)}, \dots, x_{\sigma(q-r)}, \overrightarrow{s_r})f(\overleftarrow{s_r}, x_{\sigma(q-r+1)}, \dots, x_{\sigma(2q-2r)})\\
        &\qquad\times f(x_{\rho(1)}, \dots, x_{\rho(q-r)}, \overrightarrow{s'_r})f(\overleftarrow{s'_r}, x_{\rho(q-r+1)}, \dots, x_{\rho(2q-2r)})\ d{\bf s}d{\bf s'}d{\bf x}.
    \end{align*}
    Here, we introduce the notations. For $\sigma^{-1}(1)=i\in \{ 1, \dots, 2q-2r \}$, we define 
    $$\sigma' =
\left( \begin{array}{ccc}
1 & \overset{i}{\Check{\cdots}} & 2q-2r \\
\sigma(1) & \overset{i}{\Check{\cdots}} & \sigma(2q-2r) \\
\end{array}
\right)\in S_{2q-2r-1}.$$
It should be noted, though not explicitly written, that $\sigma'$ depends on $i$ as well. Under this correspondence, there is a bijection given by
$$
\begin{array}{ccc}
    S_{2q-2r-1}\times\{ 1, \dots, 2q-2r \}&\longrightarrow&S_{2q-2r-1}\\
    (\sigma', i)&\longmapsto&\sigma.
\end{array}
$$
For $\rho$, we similarly define $(\rho',j)$. Let $\sgn(\sigma')$ denote the sign when permuting $\sigma(1), \overset{i}{\Check{\dots}}, \sigma(2q-2r)$ in ascending order. It holds that $\sgn(\sigma)=(-1)^{i+1}\sgn(\sigma')$. We define $\sigma''$ as the permutation obtained by reversing the order of $\sigma'$ as follows:
$$\sigma'' =
\left( \begin{array}{ccc}
1 & \overset{i}{\Check{\cdots}} & 2q-2r \\
\sigma(2q-2r) & \overset{2q-2r-i+1}{\Check{\cdots}} & \sigma(1) \\
\end{array}
\right)\in S_{2q-2r-1.}$$
The sign of $\sigma''$ is defined in the same way as $\sgn(\sigma')$.
It holds that $$\sgn(\sigma'')=(-1)^{(2q-2r-2)(2q-2r-1)/2}\sgn(\sigma')=(-1)^{q+r+1}\sgn(\sigma').$$ Under these preparations, we separate the pairs $\sigma, \rho$ into the following four patterns:
\begin{enumerate}
    \item $\sigma^{-1}(1), \rho^{-1}(1)\in \{ 1, \dots, q-r \}$;
    \item $\sigma^{-1}(1)\in \{ q-r+1, \dots, 2q-2r \},\ \rho^{-1}(1)\in\{ 1, \dots, q-r \}$;
    \item $\sigma^{-1}(1)\in\{ 1, \dots, q-r \},\ \rho^{-1}(1)\in \{ q-r+1, \dots, 2q-2r \}$;
    \item $\sigma^{-1}(1), \rho^{-1}(1)\in \{ q-r+1, \dots, 2q-2r \}$.
\end{enumerate}
Accordingly,
$$
\sum_{\sigma,\rho\in S_{2q-2r}}=\sum_{(\mathrm{i})}+\sum_{(\mathrm{ii})}+\sum_{(\mathrm{iii})}+\sum_{(\mathrm{iv})}.
$$
We first compute the part corresponding to $\sum_{(\mathrm{i})}$.
\begin{align*}
    &\frac{1}{((2q-2r)!)^2}\sum_{(\mathrm{i})} \sgn(\sigma)\sgn(\rho)\\
    &\quad\times \int f(x_{\sigma(1)}, \dots, x_{\sigma(i-1)}, x_1, x_{\sigma(i+1)} \dots, x_{\sigma(q-r)}, \overrightarrow{s_r})f(\overleftarrow{s_r}, x_{\sigma(q-r+1)}, \dots, x_{\sigma(2q-2r)})\\
    &\qquad\times f(x_{\rho(1)}, \dots, x_{\rho(j-1)}, x_1, x_{\rho(j+1)} \dots, x_{\rho(q-r)}, \overrightarrow{s'_r})f(\overleftarrow{s'_r}, x_{\rho(q-r+1)}, \dots, x_{\rho(2q-2r)})\ d{\bf s}d{\bf s'}d{\bf x}\\
    &=\frac{1}{((2q-2r)!)^2}\sum_{(\mathrm{i})} (-1)^{i+1}(-1)^{j+1}\sgn(\sigma')\sgn(\rho')\\
    &\quad\times \int (-1)^{i-1}f(x_1, x_{\sigma'(1)}, \dots, x_{\sigma'(q-r)}, \overrightarrow{s_r})f(\overleftarrow{s_r}, x_{\sigma'(q-r+1)}, \dots, x_{\sigma'(2q-2r)})\\
    &\qquad\times (-1)^{j-1}f(x_1, x_{\rho'(1)}, \dots, x_{\rho'(q-r)}, \overrightarrow{s'_r})f(\overleftarrow{s'_r}, x_{\rho'(q-r+1)}, \dots, x_{\rho'(2q-2r)})\ d{\bf s}d{\bf s'}d{\bf x}\\
    &=\frac{1}{(2q-2r)^2}\frac{1}{((2q-2r-1)!)^2}\sum_{\sigma,\rho \in S_{2q-2r-1}} \sgn(\sigma)\sgn(\rho)\\
    &\quad\times \int f(t, x_{\sigma(1)}, \dots, x_{\sigma(q-r-1)}, \overrightarrow{s_r})f(\overleftarrow{s_r}, x_{\sigma(q-r)}, \dots, x_{\sigma(2q-2r-1)})\\
    &\qquad\times f(t, x_{\rho(1)}, \dots, x_{\rho(q-r-1)}, \overrightarrow{s'_r})f(\overleftarrow{s'_r}, x_{\rho(q-r)}, \dots, x_{\rho(2q-2r-1)})\ d{\bf s}d{\bf s'}d{\bf x}dt\\
    &=\frac{1}{(2q-2r)^2}\int\| f(t,\cdot)\hat{\wedge}_r f \|^2_{2q-2r-1}\ dt.
\end{align*}
For $\sum_{(\mathrm{ii})}$,
\begin{align*}
    &\frac{1}{((2q-2r)!)^2}\sum_{(\mathrm{ii})} \sgn(\sigma)\sgn(\rho)\\
    &\quad\times \int f(x_{\sigma(1)}, \dots, x_{\sigma(q-r)}, \overrightarrow{s_r})f(\overleftarrow{s_r}, x_{\sigma(q-r+1)}, \dots, x_{\sigma(i-1)}, x_1, x_{\sigma(i+1)} \dots, x_{\sigma(2q-2r)})\\
    &\qquad\times f(x_{\rho(1)}, \dots, x_{\rho(j-1)}, x_1, x_{\rho(j+1)} \dots, x_{\rho(q-r)}, \overrightarrow{s'_r})f(\overleftarrow{s'_r}, x_{\rho(q-r+1)}, \dots, x_{\rho(2q-2r)})\ d{\bf s}d{\bf s'}d{\bf x}\\
    &=\frac{1}{((2q-2r)!)^2}\sum_{(\mathrm{ii})} (-1)^{i+1}(-1)^{j+1}\sgn(\sigma')\sgn(\rho')\\
    &\quad\times \int f(x_{\sigma'(1)}, \dots, x_{\sigma'(q-r)}, \overrightarrow{s_r})(-1)^{2q-2r-i}f(\overleftarrow{s_r}, x_{\sigma'(q-r+1)}, \dots, x_{\sigma'(2q-2r)}, x_1)\ \\
    &\qquad\times (-1)^{j-1}f(x_1, x_{\rho'(1)}, \dots, x_{\rho(q-r)}, \overrightarrow{s'_r})f(\overleftarrow{s'_r}, x_{\rho'(q-r+1)}, \dots, x_{\rho(2q-2r)})\ d{\bf s}d{\bf s'}d{\bf x}\\
    &=\frac{1}{((2q-2r)!)^2}\sum_{(\mathrm{ii})} (-1)^{i+1}(-1)^{j+1}(-1)^{2q-2r-i}(-1)^{j-1}(-1)^{q+r+1}\sgn(\sigma'')\sgn(\rho')\\
    &\quad\times \int f(x_1, x_{\sigma''(1)}, \dots, x_{\sigma''(q-r)}, \overrightarrow{s_r})f(\overleftarrow{s_r}, x_{\sigma''(q-r+1)}, \dots, x_{\sigma''(2q-2r)}, x_1)\\
    &\qquad\times f(x_1, x_{\rho'(1)}, \dots, x_{\rho(q-r)}, \overrightarrow{s'_r})f(\overleftarrow{s'_r}, x_{\rho'(q-r+1)}, \dots, x_{\rho(2q-2r)})\ d{\bf s}d{\bf s'}d{\bf x}\\
    &=\frac{(-1)^{q+r}}{(2q-2r)^2}\int\| f(t,\cdot)\hat{\wedge}_r f \|^2_{2q-2r-1}\ dt.
\end{align*}
Similarly, the sums $\sum_{(\mathrm{iii})},\ \sum_{(\mathrm{iv})}$ are calculated to be 
\begin{gather*}
    \frac{(-1)^{q+r}}{(2q-2r)^2}\int\| f(t,\cdot)\hat{\wedge}_r f \|^2_{2q-2r-1}\ dt,\\ 
    \frac{1}{(2q-2r)^2}\int\| f(t,\cdot)\hat{\wedge}_r f \|^2_{2q-2r-1}\ dt,
\end{gather*}
respectively. Hence, we obtain $(\ref{anti-symmetric norm eq.})$.
\end{proof}
\begin{lem}\label{fourth moment lem}
    Let $q \ge 2$ and let $f \in \Lambda_q$ be a real-valued function such that $\overleftarrow{f} = f$. Set $F = J_q(f)$. Then we have
    \begin{align*}
        m(F^4) &= 2m(F^2)^2 + \frac{(-1)^q}{2} (2q)! \int \| f(t, *) \wedge f \|^2_{2q-1} \ dt\\
        &\quad + \sum^{q-1}_{r=1} (r!)^2 (2q-2r)! \binom{p}{r}^4 \frac{1}{q} \left\{ r \frac{1+(-1)^{q+r}}{(q-r)^2} + \frac{(-1)^{q+r}}{2} (q-r) \right\} \int \| f(t, *) \hat{\wedge}_r f \|^2_{2q-2r-1} \ dt.
    \end{align*}
\end{lem}
\begin{proof}
    Note that 
    \begin{gather*}
    \mathcal{D}_t(F^3) = \mathcal{D}_t (F) F^2 + (-1)^q F \mathcal{D}_t (F) F + F^2 \mathcal{D}_t (F),\\
    m(\delta(F_t) G) = \int m( \beta(F_t) \mathcal{D}_t (G) )\ dt. 
    \end{gather*}
    Then
    \begin{align*}
    m(F^4) &= m\big(\mathcal{R} \circ \mathcal{R}^{-1}(F)F^3\big)\\
    &= \frac{1}{q} m\big(\delta \circ \mathcal{D}_t (F) F^3\big)\\
    &= \frac{(-1)^{q-1}}{q} \int m\big( \mathcal{D}_t (F) \mathcal{D}_t (F^3) \big)\ dt\\
    &= \frac{(-1)^{q-1}}{q} \int m\big( \mathcal{D}_t (F) \mathcal{D}_t (F) F^2 + (-1)^q \mathcal{D}_t (F) F \mathcal{D}_t (F) F + \mathcal{D}_t (F) F^2 \mathcal{D}_t (F) \big) \ dt\\
    &= \frac{(-1)^{q-1}}{q} \int \left\{2 m\big((\mathcal{D}_t (F))^2 F^2\big) + (-1)^q m\big( (\mathcal{D}_t (F) F)^2 \big)\right\}\ dt.
\end{align*}
In the last equality, the cyclicity of $m$ was used.
We calculate $ m\big((\mathcal{D}_t (F))^2 F^2\big) $. As a preliminary step, we perform calculations similar to those in the proof of Lemma~\ref{wedge lemma}.
\begin{align*}
    \int f(t, *) &\wedge_r f(t, *) (x_1, \dots, x_{2q-2-2r})\ dt\\
    &= \int f( t, x_1, \dots, x_{q-1-r}, \overrightarrow{s_r} ) f( t, \overleftarrow{s_r}, x_{q-r}, \dots, x_{2q-2-2r} )\ d{\bf s} dt\\
    &= \int (-1)^{q-1} f( x_1, \dots, x_{q-1-r}, \overrightarrow{s_r}, t ) f( t, \overleftarrow{s_r}, x_{q-r}, \dots, x_{2q-2-2r} )\ d{\bf s} dt\\
    &= (-1)^{q-1} f \wedge_{r+1} f.
\end{align*}
Thus, following the initial steps in the proof of Theorem~\ref{claim1}, we have
\begin{align*}
    \int ( \mathcal{D}_t (F) )^2 \ dt &= (-1)^{q-1} q^2 \sum^q_{r=1} (r-1)! \binom{q}{r}^2 J_{2q-2r}(f \hat{\wedge}_r f), \\
    F^2 &= \sum^q_{r=0} r! \binom{q}{r}^2 J_{2q-2r}(f \hat{\wedge}_r f).
\end{align*}
Therefore, 
\begin{align*}
    \int m\big( (\mathcal{D}_t (F) )^2 F^2\big)\ dt &= \int m\big((F^2)^* (\mathcal{D}_t (F) )^2 \big)\ dt\\
    & = (-1)^{q-1} q^2 \sum^q_{r=1} r! (r-1)! \binom{q}{r}^2 \binom{q-1}{r-1}^2 (2q-2r)! \| f \hat{\wedge}_r f \|^2_{2q-2r}.
\end{align*}
Next, we calculate $\int m( ( \mathcal{D}_t (F) F )^2 )\ dt$. Noting that $q(q-1) \equiv 0 \bmod 2$ and using Lemma~\ref{wedge lemma}, we obtain
\begin{align*}
    \overline{\overleftarrow{f}} \wedge_r \overline{f}( t, \overleftarrow{*} )
    &= f \wedge_r f(*, t) = (-1)^{r} f(*, t) \wedge_r f \\
    &= (-1)^r (-1)^{q-1} f(t, *) \wedge_r f.
\end{align*}
Given the above, we have
\begin{align*}
    F^* \mathcal{D}_t(F)^* 
    &= q \sum^{q-1}_{r=0} r! \binom{q}{r} \binom{q-1}{r} J_{2q-1-2r}(\overline{\overleftarrow{f}} \wedge_r \overline{f}( t, \overleftarrow{*} ))\\
    &= q \sum^{q-1}_{r=0} r! \binom{q}{r} \binom{q-1}{r} (-1)^{q-1+r}J_{2q-1-2r}(f(t, *) \hat{\wedge}_r f).
\end{align*}
Therefore,
\begin{align*}
    \int m\big( ( \mathcal{D}_t (F) F )^2 \big)\ dt &= \int m\big( ( F^* \mathcal{D}_t(F)^* )^* ( \mathcal{D}_t (F) F ) \big)\ dt\\
    &= q^2 \sum^{q-1}_{r=0} (-1)^{r+1-q} (r!)^2 \binom{q}{r}^2 \binom{q-1}{r}^2 (2q-2r-1)! \int \|f(t, *) \hat{\wedge}_r f\|^2_{2q-2r-1}\ dt.
\end{align*}
Thus,
\begin{align*}
    m(F^4) &= \frac{(-1)^{q-1}}{q} \left\{ 2 (-1)^{q-1} q^2 \sum^q_{r=1} r! (r-1)! \binom{q}{r}^2 \binom{q-1}{r-1}^2 (2q-2r)! \| f \hat{\wedge}_r f \|^2_{2q-2r} \right.\\
    &\left. \quad + (-1)^q q^2 \sum^{q-1}_{r=0} (-1)^{r+1-q} (r!)^2 \binom{q}{r}^2 \binom{q-1}{r}^2 (2q-2r-1)! \int \|f(t, *) \hat{\wedge}_r f\|^2_{2q-2r-1}\ dt \right\}\\
    &= 2q \sum^q_{r=1} r! (r-1)! \binom{q}{r}^2 \binom{q-1}{r-1}^2 (2q-2r)! \| f \hat{\wedge}_r f \|^2_{2q-2r} \\
    &\quad + q \sum^{q-1}_{r=0} (-1)^{q+r} (r!)^2 \binom{q}{r}^2 \binom{q-1}{r}^2 (2q-2r-1)! \int \|f(t, *) \hat{\wedge}_r f\|^2_{2q-2r-1}\ dt\\
    &= 2 \sum^{q-1}_{r=1} (r!)^2 \binom{q}{r}^4 \frac{r}{q} (2q-2r)! \frac{1+(-1)^{q+r}}{2(q-r)^2} \int \|f(t, *) \hat{\wedge}_r f\|^2_{2q-2r-1}\ dt\\
    &\quad + \sum^{q-1}_{r=1} (r!)^2 \binom{q}{r}^4 \frac{(-1)^{q+r}}{4q} (2q-2r) (2q-2r)! \int \|f(t, *) \hat{\wedge}_r f\|^2_{2q-2r-1}\ dt\\
    &\qquad+ 2m(F^2)^2 + \frac{(-1)^q}{2} (2q)! \int \|f(t, *) \wedge f\|^2_{2q-1}\ dt,
\end{align*}
where we used $\binom{q-1}{r-1} = \frac{r}{q} \binom{q}{r}, \binom{q-1}{r} = \frac{q-r}{q} \binom{q}{r}$, and Lemma~\ref{anti-symmetric norm} for the last equality.
\end{proof}
In the usual fourth moment theorem, the counterpart of
$$
2m(F^2)^2 + \frac{(-1)^q}{2} (2q!) \int \| f(t, *) \wedge f \|^2_{2q-1} \ dt
$$
that appears in the above lemma is $3E[F^2]^2$ {\cite[Lem.~5.2.4]{NP12}}.
Hence, it seems natural to think of the $\sum^{q-1}_{r=1}$ part in Lemma~\ref{fourth moment lem} as the counterpart of $E[F^4]-3E[F^2]^2$ on the Clifford algebra. 
The following theorem shows that it is difficult to handle convergence with the fourth moment.
\begin{thm}\label{claim2}
    Let $f_1, \dots, f_4 \in L^2(\mathbb{R}_+)$ be non-zero real valued functions such that $\langle f_i, f_j \rangle_{L^2(\mathbb{R}_+)} =0$ for $ i \neq j$. We set $ f = f_1\wedge \dots \wedge f_4$, $ F=J_4(f) $, and
    \begin{equation*}
        K(F) \coloneq m(F^4) - 2m(F^2)^2 - \frac{(-1)^q}{2} (2q!) \int \| f(t, *) \wedge f \|^2_{2q-1} \ dt.
    \end{equation*}
    Then we have
    \begin{enumerate}
    \item $K(F) \neq 0 $,
    \item $\| f \hat{\wedge}_2 f \|_4 = 0.$
    \end{enumerate}
\end{thm}
\begin{proof}
    From Lemma~\ref{fourth moment lem}, we have
    \begin{align*}
        K(F) &= (1!)^2\cdot 4^4\cdot 6!\cdot \frac{1}{4} \left\{ 1\cdot\frac{0}{3^2} + \frac{-1}{2}\cdot 3 \right\} \int \| f(t, *) \hat{\wedge}_1 f \|^2_5 \ dt \\
        &\quad + (2!)^2\cdot 6^4\cdot 4!\cdot \frac{1}{4} \left\{ 2\cdot\frac{2}{2^2} + \frac{1}{2}\cdot 2 \right\} \int \| f(t, *) \hat{\wedge}_2 f \|^2_3 \ dt\\
        &\qquad + (3!)^2\cdot 4^4\cdot 2!\cdot \frac{1}{4} \left\{ 3\cdot\frac{0}{1^2} + \frac{-1}{2}\cdot 1 \right\} \int \| f(t, *) \hat{\wedge}_3 f \|^2_1 \ dt.
    \end{align*}
    Therefore, in order to prove Theorem~\ref{claim2}, it is sufficient to show
    \begin{gather}
     f(t, *) \hat{\wedge}_2 f  = 0 \quad\text{for}\ \mathrm{a.e.}\ t\ge 0,\label{1}\\
    \int \| f(t, *) \hat{\wedge}_3 f \|^2_1 \ dt \neq 0.\label{2}
\end{gather}
First, we prove (\ref{2}):
\begin{align*}
    \int \| f(t, *) \hat{\wedge}_3 f \|^2_1 \ dt 
    &= \int f( t, \overrightarrow{s_3} ) f( \overleftarrow{s_3}, x ) 
    f( t, \overrightarrow{s'_3} ) f( \overleftarrow{s'_3}, x )\ d{\bf s}d{\bf s'}dxdt,\\
    (4!)^2 f(t, *) \hat{\wedge}_3 f (x)
    &= (4!)^2 \int f( t, \overrightarrow{s_3} ) f( \overleftarrow{s_3}, x )\ d{\bf s}\\
    &= (4!)^2 \int f( \overleftarrow{s_3}, t ) f( \overleftarrow{s_3}, x )\ d{\bf s}.
\end{align*}
Furthermore,
\begin{align*}
    4!f(x_1, x_2, x_3, x_4) &= \sum_{\sigma \in S_4} \sgn(\sigma) f_1(x_{\sigma(1)})f_2(x_{\sigma({2})})f_3(x_{\sigma({3})})f_4(x_{\sigma({4})})\\
    &= \sum^4_{i=1} \sum_{\sigma: \sigma(i)=4} \sgn(\sigma) f_1(x_{\sigma(1)}) \cdots f_i(x_4) \cdots f_4(x_{\sigma(4)}).
\end{align*}
Given that $\langle f_i, f_j \rangle_{L^2(\mathbb{R}_+)}=0$ for $i \neq j$,
$$
    (4!)^2 f(t, *) \hat{\wedge}_3 f (x) = 3! \sum^4_{i=1} \| f_1 \|^2 \overset{i}{\Check{\cdots}} \| f_4 \|^2 f_i(t) f_i(x).
    $$
    Therefore
    $$
    \int \| f(t, *) \hat{\wedge}_3 f\|^2_1\ dt = \frac{(3!)^2 \cdot 4}{(4!)^4} \| f_1 \|^4 \cdots \| f_4 \|^4 \neq 0.$$
Next, we prove (\ref{1}).
\begin{align*}
    f(t, *) \hat{\wedge}_2 f 
    &= \frac{1}{3!} \int \sum_{\sigma \in S_3} \sgn(\sigma) f(t, x_{\sigma(1)}, \overrightarrow{s_2}) f( \overleftarrow{s_2},x_{\sigma(2)}, x_{\sigma(3)} )\ d{\bf s}\\
    &= \frac{1}{3!} \int \sum_{\sigma \in S_3} \sgn(\sigma) f(\overleftarrow{s_2}, x_{\sigma(1)}, t) f( \overleftarrow{s_2},x_{\sigma(2)}, x_{\sigma(3)} )\ d{\bf s}.
\end{align*}
Let
$$
F_k = \int f(\overleftarrow{s_2}, x_{a_k (1)}, t) f( \overleftarrow{s_2},x_{a_k (2)}, x_{a_k (3)} )\ d{\bf s},
$$
where $a_k \in S_3$ are 
 \begin{align*}
a_0 &= \Big( \begin{array}{ccc}
1 & 2 & 3 \\
1 & 2 & 3 \\
\end{array}
\Big),\ a_1 = \Big( \begin{array}{ccc}
1 & 2 & 3 \\
3 & 1 & 2 \\
\end{array}
\Big),\ a_2 = \Big( \begin{array}{ccc}
1 & 2 & 3 \\
2 & 3 & 1 \\
\end{array}
\Big),\\
a_3 &= \Big( \begin{array}{ccc}
1 & 2 & 3 \\
2 & 1 & 3 \\
\end{array}
\Big),\ a_4 = \Big( \begin{array}{ccc}
1 & 2 & 3 \\
1 & 3 & 2 \\
\end{array}
\Big),\ a_5 = \Big( \begin{array}{ccc}
1 & 2 & 3 \\
3 & 2 & 1 \\
\end{array}
\Big)   .
\end{align*}
From the anti-symmetry of $f$, we have
$$
F_0 = -F_4,\ F_1 = -F_5,\ F_2 = -F_3.
$$
Thus,
\begin{align*}
    3! f(t, *) \hat{\wedge}_2 f &= F_0 + F_1 + F_2 - F_3 - F_4 - F_5\\
    &= 2( F_0 + F_1 + F_2 ).
\end{align*}
We compute $F_0, F_1$ and $ F_2$. For example, we write $g_{34}(1,0,2,3)$ for $f_3(x_1)f_4(t)f_3(x_2)f_4(x_3)$. By taking account of $\langle f_i, f_j \rangle_{L^2(\mathbb{R}_+)}=0$ for $i \neq j$, the result is as follows:
\begin{align*}
    &(4!)^2 F_0 \\
    &= \int \Big\{ f_1(s_1) f_2(s_2) f_3(x_1) f_4(t) -f_1(s_1) f_2(s_2) f_3(t) f_4(x_1) \\
    &\quad-f_1(s_2) f_2(s_1) f_3(x_1) f_4(t) +f_1(s_2) f_2(s_1) f_3(t) f_4(x_1)\\
    &\qquad + \cdots \Big\}\\
    &\times\Big\{ f_1(s_1) f_2(s_2) f_3(x_2) f_4(x_3) -f_1(s_1) f_2(s_2) f_3(x_3) f_4(x_2)\\
    &\quad-f_1(s_2) f_2(s_1) f_3(x_2) f_4(x_3) +f_1(s_2) f_2(s_1) f_3(x_3) f_4(x_2)\\
    &\qquad + \cdots \Big\}\ d{\bf s}\\
    &= 2 \| f_1 \|^2 \| f_2 \|^2 \Big\{ f_3(x_1)f_4(t)f_3(x_2)f_4(x_3) -f_3(x_1)f_4(t)f_3(x_3)f_4(x_2) \\&\quad -f_3(t)f_4(x_1)f_3(x_2)f_4(x_3) +f_3(t)f_4(x_1)f_3(x_3)f_4(x_2)  \Big\}\\
    &+ 2 \| f_1 \|^2 \| f_3 \|^2 \Big\{ f_2(x_1)f_4(t)f_2(x_2)f_4(x_3) -f_2(x_1)f_4(t)f_2(x_3)f_4(x_2) \\&\quad -f_2(t)f_4(x_1)f_2(x_2)f_4(x_3) +f_2(t)f_4(x_1)f_2(x_3)f_4(x_2)  \Big\}\\
    &+ 2 \| f_1 \|^2 \| f_4 \|^2 \Big\{ f_2(x_1)f_3(t)f_2(x_2)f_3(x_3) -f_2(x_1)f_3(t)f_2(x_3)f_3(x_2) \\&\quad -f_2(t)f_3(x_1)f_2(x_2)f_3(x_3) +f_2(t)f_3(x_1)f_2(x_3)f_3(x_2)  \Big\}\\
    &+ 2 \| f_2 \|^2 \| f_3 \|^2 \Big\{ f_1(x_1)f_4(t)f_1(x_2)f_4(x_3) -f_1(x_1)f_4(t)f_1(x_3)f_4(x_2) \\&\quad -f_1(t)f_4(x_1)f_1(x_2)f_4(x_3) +f_1(t)f_4(x_1)f_1(x_3)f_4(x_2)  \Big\}\\
    &+ 2 \| f_3 \|^2 \| f_4 \|^2 \Big\{ f_1(x_1)f_2(t)f_1(x_2)f_2(x_3) -f_1(x_1)f_2(t)f_1(x_3)f_2(x_2) \\&\quad -f_1(t)f_2(x_1)f_1(x_2)f_2(x_3) +f_1(t)f_2(x_1)f_1(x_3)f_2(x_2)  \Big\}\\
    &+ 2 \| f_2 \|^2 \| f_4 \|^2 \Big\{ f_1(x_1)f_3(t)f_1(x_2)f_3(x_3) -f_1(x_1)f_3(t)f_1(x_3)f_3(x_2) \\&\quad -f_1(t)f_3(x_1)f_1(x_2)f_3(x_3) +f_1(t)f_3(x_1)f_1(x_3)f_3(x_2)  \Big\}\\
    &= 2 \| f_1 \|^2 \| f_2 \|^2 \Big\{ g_{34}(1,0,2,3) -g_{34}(1,0,3,2) -g_{34}(0,1,2,3) +g_{34}(0,1,3,2) \Big\}\\
    &+ 2 \| f_1 \|^2 \| f_3 \|^2 \Big\{ g_{24}(1,0,2,3) -g_{24}(1,0,3,2) -g_{24}(0,1,2,3) +g_{24}(0,1,3,2) \Big\}\\
    &+ 2 \| f_1 \|^2 \| f_4 \|^2 \Big\{ g_{23}(1,0,2,3) -g_{23}(1,0,3,2) -g_{23}(0,1,2,3) +g_{23}(0,1,3,2) \Big\}\\
    &+ 2 \| f_2 \|^2 \| f_3 \|^2 \Big\{ g_{14}(1,0,2,3) -g_{14}(1,0,3,2) -g_{14}(0,1,2,3) +g_{14}(0,1,3,2) \Big\}\\
    &+ 2 \| f_3 \|^2 \| f_4 \|^2 \Big\{ g_{12}(1,0,2,3) -g_{12}(1,0,3,2) -g_{12}(0,1,2,3) +g_{12}(0,1,3,2) \Big\}\\
    &+ 2 \| f_2 \|^2 \| f_4 \|^2 \Big\{ g_{13}(1,0,2,3) -g_{13}(1,0,3,2) -g_{13}(0,1,2,3) +g_{13}(0,1,3,2) \Big\}.
\end{align*}
If we write $F_0(1,0,2,3)$ for $F_0$ for simplicity, then we also get
\begin{align*}
    &(4!)^2 F_1 \\
    &= 2 \| f_1 \|^2 \| f_2 \|^2 \Big\{ g_{34}(3,0,1,2) -g_{34}(3,0,2,1) -g_{34}(0,3,1,2) +g_{34}(0,3,2,1) \Big\}\\
    &+ \cdots\\
    &= (4!)^2 F_0(3,0,1,2),
\end{align*}
\begin{align*}
    &(4!)^2 F_2 \\
    &= 2 \| f_1 \|^2 \| f_2 \|^2 \Big\{ g_{34}(2,0,3,1) -g_{34}(2,0,1,3) -g_{34}(0,2,3,1) +g_{34}(0,2,1,3) \Big\}\\
    &+ \cdots\\
    &= (4!)^2 F_0(2,0,3,1).
\end{align*}
Using the symmetry between the first and third variables of $g_{ij}$ and that between the second and fourth variables of $g_{ij}$, where $i,j \in \{ 1,2,3,4 \}$, we obtain $F_0 + F_1 + F_2 = 0.$
\end{proof}
\begin{rem}
    Let $z\in L^2(\mathbb{R}_+)$. We can prove $K(\Psi(z))=0$.
\end{rem}

% 5
\section{Concluding Remarks}\label{section_comment}
In this paper, we dealt with Malliavin calculus on the Clifford algebra and investigated the differences between this calculus and usual Malliavin calculus or Brownian motion. In this section, We would like to discuss a few possible developments in the future.

The theory in the framework of Brownian motion for bosons includes Hudson-Parthasarathy theory \cite{Par92,Mey95}, which corresponds to It\^o calculus. This theory further developed into non-causal stochastic calculus \cite{Lin93,Mey95} and into more generalized forms \cite{FLS01}. Given that there also exists Hudson-Parthasarathy theory for fermions \cite{AH84}, it seems natural to consider extending current Malliavin calculus on the Clifford algebra in the direction of \cite{Lin93}. The key point in \cite{Lin93} was that there is an isomorphism between the symmetric Fock space and the symmetric measure space \cite{Gui}. According to {\cite[Ex.~19.14]{Par92}}, the anti-symmetric Fock space also has an isomorphism to the symmetric measure space. This implies that analyses of both bosons and fermions can be developed in the symmetric measure space. By comparing the relationships between their derivations and divergences on the symmetric measure space, we might obtain some new results. It is likely that these connections could be linked through the Jordan-Wigner transformation {\cite[Example~25.18]{Par92}}.

However, it is unclear whether the current results can be extended in a more generalized context as in \cite{FLS01}. For instance, the anti-symmetric derivation considered here does not satisfy the Leibniz rule, which means that it is not possible to generalize the derivation by using the complete positivity of shift operations as done in \cite{FLS01}. If a map is completely positive, its generator should satisfy the Leibniz rule. Notably, for bosons, the results of \cite{FLS01} might be further generalized. For example, it is known that under quasi-free CCR conditions, the operation $W(\varphi)\mapsto W(\varphi)\exp{i\langle \varphi, \psi \rangle}$ is completely positive \cite{DVV79}. In the case of quasi-free CAR, such an operation does not satisfy the condition of complete positivity \cite{Eva79}. Furthermore, it is worth noting that It\^o calculus exists for both quasi-free CCR and CAR \cite{BSW5}.

If we consider extending Malliavin calculus not only to fermions but also to $q$-Brownian motions on $L^2(\mathbb{R}_+)$, it might be possible to utilize the chaos expansion as done in this paper \cite{Spe03}. However, it is uncertain what the multiplication formula would be, so it may not work out smoothly.

Regarding applications, this paper considered three examples to which usual Malliavin calculus is applied: the concentration inequality, the logarithmic Sobolev inequality, and the fourth-moment theorem. We investigated how the proofs and conclusions would change in the fermion setting. Through the investigation, the author faced the inconvenience of the Leibniz rule not holding many times. While 
the concentration inequality that does not depend on the Leibniz rule was manageable, dealing with anti-symmetric derivative of $\log$ was beyond his capability. The fourth-moment theorem, depending on getting the anti-symmetric derivative of $x^3$, required very lengthy calculations. Nevertheless, the obtained result is interesting.

Currently, the author is exploring whether L\'{e}vy's stochastic area can be represented on the anti-symmetric Fock space. It is expected that L\'{e}vy's stochastic area is represented by using the anti-symmetric divergence and we 
can utilize CAR for the moment calculations.

\section*{Acknowledgment}
I would like to express my heartfelt gratitude to Professor Yuu Hariya who is my supervisor for his continued support, invaluable advice, and unwavering encouragement. This paper would not have been completed without his guidance and persistent assistance.

\bibliographystyle{alpha}
\bibliography{ref}

\end{document}